\documentclass{amsart}
\usepackage{amssymb}
\usepackage{amsfonts}
\usepackage{amssymb}
\usepackage{amsmath} 
\usepackage{amsthm}
\usepackage{enumerate}
\usepackage{tabularx}
\usepackage{centernot}
\usepackage{mathtools}
\usepackage{stmaryrd}
\usepackage{amsthm,amssymb}
\usepackage{etoolbox,color}
\usepackage{tikz}
\usepackage{amssymb}
\usetikzlibrary{matrix}
\usepackage{tikz-cd}
\usepackage{tikz}
\usepackage{marginnote}
\definecolor{mygray}{gray}{0.85}
\usepackage[linecolor=black, backgroundcolor=mygray,colorinlistoftodos,prependcaption,textsize=small]{todonotes}
\usepackage{xargs}  
\usepackage[colorlinks,citecolor=blue,urlcolor=blue, linkcolor=blue]{hyperref}
\newcommand\+[1]{\mathcal{#1}}

\renewcommand{\leq}{\leqslant}
\renewcommand{\geq}{\geqslant}

\newcommand{\mrm}[1]{\mathrm{#1}}

\usepackage[backgroundcolor=mygray,colorinlistoftodos,prependcaption,textsize=small]{todonotes}
\usepackage{xcolor}

\newcommand{\todog}[1]{\todo[inline, color=mygray]{#1}}

\def\sub{\subseteq}
\newcommand{\Aut}{\mrm{Aut}}
\newcommand{\Inn}{\mrm{Inn}}
\newcommand{\Out}{\mrm{Out}}
\newcommand{\Sym}{\mrm{Sym}}
\def\aaa{\alpha}
\def\sss{\sigma}

\def\LA{\Leftarrow}
\def\RA{\Rightarrow}
\def\ol{\bar} 
\def\lra{\leftrightarrow}
\def\n{\noindent}
\makeatletter
\def\subsection{\@startsection{subsection}{3}%
  \z@{.5\linespacing\@plus.7\linespacing}{.3\linespacing}%
  {\bfseries\centering}}
\makeatother

\makeatletter
\def\subsubsection{\@startsection{subsubsection}{3}%
  \z@{.5\linespacing\@plus.7\linespacing}{.3\linespacing}%
  {\centering}}
\makeatother

\makeatletter
\def\myfnt{\ifx\protect\@typeset@protect\expandafter\footnote\else\expandafter\@gobble\fi}
\makeatother

\newcommand{\andre}[1]{{\textcolor{magenta}{Andre: #1}} }

\newtheorem{theorem}{Theorem}[section]

\newtheorem{convention}[theorem]{Conventions}
\newtheorem{corollary}[theorem]{Corollary}

\newtheorem{lemma}[theorem]{Lemma}
\newtheorem{proposition}[theorem]{Proposition}
\newtheorem{example}[theorem]{Example}

\theoremstyle{plain}
\newtheorem{thm}[theorem]{Theorem}

\newtheorem{cor}[theorem]{Corollary}

\theoremstyle{definition}

\newtheorem{fact}[theorem]{Fact}
\newtheorem{definition}[theorem]{Definition}
\newtheorem{remark}[theorem]{Remark}

\newtheorem{notation}[theorem]{Notation}
\newtheorem{assumption}[theorem]{Assumption}

\newcounter{claimcounter}
\newenvironment{claim}{\stepcounter{claimcounter}{\noindent {\bf Claim \theclaimcounter.}}}{}

\newcommand{\acl}{\mathrm{acl}}
\newcommand{\dcl}{\mathrm{dcl}}

\begin{document}
%%%%%%%%%%%%%%%%%%
%\begin{abstract} We continue the work from \cite{reconstruction2} on the search for dividing lines on oligomorphic groups toward progress on the question of smoothness of topological isomorphism on oligomorphic groups \cite{coarse}. We isolate two main dividing lines which both ensure smoothness of topological isomorphism, namely: (i) no algebraicity and (ii) the existence of finitely many {\em essential} subgroups up to conjugacy. This second dividing line is new and it includes all $\aleph_0$-categorical structures which have weak elimination of imaginaries after adding finitely many imaginary sorts. %We actually conjecture that all $\mathfrak{G}$-finite groups belong to the second dividing line; this points to the fact that a possible reduction of $E_0$ to topological isomorphism on oligomorphis groups has to involve somehow the most exotic $\aleph_0$-categorical structures present in the literature, i.e., finite covers of non $\mathfrak{G}$-finite oligomorphic groups.
%In the process, we also prove results of independent interest on outer topological automorphisms of $G$, for $G$ an oligomorphic group satisfying either one of our dividing lines, specifically we prove that $\mrm{Aut}(G)$ is a quasi-oligomorphic group and that $\mrm{Aut}(G)/\mrm{Inn}(G)$ is a profinite group.
%\end{abstract}
\title[Oligomorphic groups]{Oligomorphic groups,  their automorphism groups,   and  the  complexity of their isomorphism}
%\title {Symmetries and isomorphisms  for oligomorphic groups}
%\title {Topological isomorphisms between oligomorphic groups}
%\title[Dividing lines on the isomorphism problem for oligomorphic groups]{A quest for dividing lines on the isomorphism problem for oligomorphic groups}

\begin{abstract} 
The paper establishes results following  two interconnected directions. 

1.  Let $G$ be a    Roelcke precompact closed subgroup of the group $\mathrm{Sym}(\omega)$ of permutations of the natural numbers.  Let $\mathrm{Aut}(G)$ denote the group of continuous automorphisms of $G$. Then  $\mathrm{Inn}(G)$ is closed in $\mathrm{Aut}(G)$, where $\mathrm{Aut}(G)$ carries the topology of pointwise convergence for its (faithful) action on the    cosets of open subgroups.  Under the stronger hypothesis that~$G$ is oligomorphic, $\+ N_G/G$ is profinite, where $\+ N_G$ denotes the normaliser of~$G$ in $\mathrm{Sym}(\omega)$, and the topological group $\mathrm{Out}(G)= \mathrm{Aut}(G)/\mathrm{Inn}(G)$ is totally disconnected,  locally compact. 
 
2a. We provide a general method to show smoothness of the isomorphism relation for appropriate  Borel classes of oligomorphic groups. We  apply it to two such classes: the  oligomorphic  groups with   no algebraicity, and the oligomorphic groups with finitely many {essential} subgroups up to conjugacy.  

2b. Using this method we also show that if $G$ is in such a Borel  class, then $\mathrm{Aut}(G)$ is topologically isomorphic to an oligomorphic group, and   $\mathrm{Out}(G)$ is profinite.
\end{abstract}

\thanks{The first author was supported by the Marsden fund of New Zealand, grant number UOA-1931. The second author was  supported by project PRIN 2022 ``Models, sets and classifications", prot. 2022TECZJA, and by INdAM Project 2024 (Consolidator grant) ``Groups, Crystals and Classifications''.}

\author{Andr{\'e} Nies}
\address{School of Computer Science, The University of Auckland, New Zealand}
\email{andre@cs.auckland.ac.nz}

\author{Gianluca Paolini}
\address{Department of Mathematics ``Giuseppe Peano'', University of Torino, Italy.}
\email{gianluca.paolini@unito.it}

%\author{...}
%\address{...}
%\email{...}

\date{\today}
\maketitle

\tableofcontents

\section{Introduction}
\noindent {\it Background.} A  {permutation group}  on an infinite set $X$ is a subgroup of the group $\Sym(X) $ of permutations of $X$. Such a group  is called   \emph{oligomorphic} (Cameron~\cite{cameron_oligo}) if for each $n \in \omega$, the canonical action of $G$ on $X^n$ has only finitely many orbits.  
The \emph{closed}  subgroups of $\Sym(X)$ are precisely the automorphism groups of  structures with domain  $X$. These are topological groups: a neighbourhood basis of $1$ is given by the pointwise stabilisers of finite sets.     Oligomorphic closed subgroups $G$ of $\Sym(\omega)$ correspond exactly to automorphism groups of  $\aleph_0$-categorical  structures with domain $\omega$, structures of crucial importance in model theory. Concretely, we have $G= \Aut(M_G)$ where $M_G$ is the structure on $\omega$ that has an $n$-ary relation for each $n$-orbit of $G$.  Two oligomorphic groups are topologically isomorphic  if and only if the corresponding structures are bi-interpretable (Coquand, see~\cite{biinte}).
 
 The  close link between the topological study of oligomorphic groups and the model theory of $\aleph_0$-categorical  structures will determine our perspective of oligomorphic groups in this paper. However, these groups have been studied from a wide variety of other angles. For instance, Cameron studied the  possible growth of the number of orbits, both of $n$-element sets, and of $n$-tuples of pairwise distinct elements. This   connects oligomorphic groups  to the area of combinatorial enumeration. In theoretical computer science, they play a role in constraint satisfaction problems, when   the   templates  are certain reducts of infinite homogeneous structures~\cite{Bodirsky.Nesetril:06}.
    
  \medskip
  
\noindent {\it Results.} We   follow two interrelated directions.  The \emph{first direction}  studies, for the first time in generality,  the  groups $\Aut(G)$ of topological automorphisms of an oligomorphic closed subgroup $G$ of $\Sym(\omega)$,  and its  outer automorphism group $\Out(G)$.  The  group  $\Aut(G)$ has a unique Polish   topology, recalled  in \ref{df:nbhd} below,    that makes its  action on $G$ continuous  (\cite{Nies.Schlicht:nd}; also  see~\cite[Section~6]{LogicBlog:23}). The notation $H \leq_c \Sym(\omega)$ indicates that $H$ is a closed subgroup of $\Sym(\omega)$. 
%, which we assume below.

\begin{theorem} \label{th:summary part 1} Let $G \leq_c \Sym(\omega)$ be an oligomorphic group. Let $\+ N_G$ denote its normaliser in $\Sym(\omega)$.
	\begin{enumerate}[(a)]
 \item $\Inn(G)$ is closed in $\Aut(G)$.
 \item The Polish group $\+ N_G/G$ is profinite.
 \item The Polish group $\Out(G)$ is totally disconnected, locally compact (t.d.l.c.).  \end{enumerate}
\end{theorem}  
This theorem is obtained through  three  separate   results. Part  (a) is proved for the wider setting of Roelcke precompact groups in \ref{prop: Inn closed}. Part (b) follows from \ref{profinite_lemma}, and part (c) is obtained through \ref{prop: tdlc}.
The ``upper bound" on the complexity of $\Out(G)$ in (c)  begs the   question
whether  $\Out(G)$ is always profinite for an oligomorphic group $G$.
It is easy to see from our proofs that the map $G \mapsto \Out(G)$ is a Borel invariant of oligomorphic groups. Our results open the door for a wider study of oligomorphic groups via this invariant, perhaps similar to the case of finite groups~\cite{Wilson:25}. 

 The \emph{second direction} is to study   the complexity, in the sense of Borel reducibility~$\leq_B$ (see \cite{gao}), of the topological  isomorphism relation between oligomorphic groups. The set of  oligomorphic groups can be seen as a standard Borel space (see~\cite{nies}). In~\cite{coarse} it was shown that their isomorphism relation is   Borel reducible to    a Borel equivalence relation with all classes countable;   it was left open whether this equivalence relation is in fact much lower: is it Borel   reducible to the identity relation on $\mathbb R$?  Such an equivalence relation is called \emph{smooth}. 

We note that the research for classes of oligomorphic groups is part of a larger research programme started in \cite{nies}:  to determine the complexity of the  topological isomorphism relation on natural  Borel classes of closed subgroups of $\Sym(\omega)$.   For example, topological isomorphism on the class of topologically \emph{finitely generated} compact subgroups of $\Sym(\omega)$  is smooth:   such groups are given by the collection of isomorphism types of their finite quotients~\cite{Jarden.Lubotzky:08}, and  the set of   types of such   quotients can be seen as an element of a standard Borel space, when suitably encoded as a set of natural numbers. In contrast, isomorphism on the class of \emph{all} compact subgroups of $\Sym(\omega)$ is Borel equivalent to the isomorphism relation between countable graphs, and hence not smooth \cite[Thm.\ 4.3]{nies}.  (We note that the compact closed subgroups of $\Sym(\omega)$ are, up to topological isomorphism, precisely the countably based profinite groups.) 
 
 We provide  in Theorem~\ref{second_th-rec} a sufficient criterion for when the isomorphism relation on a Borel class  $\+ V$ of oligomorphic groups is smooth (in brief, we will say that $\+ V$ is smooth). The criterion says that there is a  Borel and invariant property $P$ picking out open subgroups from groups in $\+ V$, so that the filter of all open subgroups of a group $G$ in $\+ V$ is  generated by finitely many  conjugacy classes of subgroups enjoying~$P$. We  apply the criterion to two classes.
 
The first class consists of  the automorphism groups  of $\aleph_0$-categorical structures with no algebraicity. Here the property  $P$ says that the subgroup is the stabiliser of a number. 

The second class consists of the oligomorphic groups  with finitely many essential subgroups up to conjugacy.  
Here, an \emph{almost  essential} subgroup is an open subgroup~$U$ without a proper subgroup of finite index such that the   conjugacy class  of~$U$ generates the filter of open subgroups; to be \emph{essential} means additionally  to be   of minimal depth, where the depth of an open subgroup $U$ is the least length of a maximal  chain of subgroups above~$U$ (noting that there are only finitely many subgroups of $G$ above $U$).    Definition~\ref{def_essential_intro} will provide the detail.  %The property $P$ says that the subgroup is {essential}. 

  \begin{theorem} \label{great theorem}Let $\+ C$ be either
 	
 	\n (1)  the class of oligomorphic groups with no algebraicity, or  
 	
 	\n  (2) the class of oligomorphic groups with finitely many essential subgroups up to conjugacy.
 	\begin{enumerate}[(a)] \item The topological  isomorphism relation on $\+ C$ is smooth.
 		\item If $G \in \+ C$ then $\Out(G)$ is profinite, and $\Aut(G)$ is isomorphic to an oligomorphic group. 
 		   \end{enumerate}
 \end{theorem} 
 The results are obtained by verifying  the hypotheses that are needed for the general  Theorem~\ref{second_th-rec}: for (1)  see Lemma~\ref{lem: no alg}; for (2) see Lemma~\ref{lem:essential}.

The groups $\Aut(G)$ and $\Out(G)$ have somewhat more concrete descriptions when  $G$ is  in  the first class, i.e., $G = \mathrm{Aut}(M)$, for a structure  $M$ with no algebraicity. Let $\mathcal{E}_M$  denote    the orbital structure of $M_G$, which is the  reduct of $M_G$ that for each $n$ has the equivalence relation of being in the same $n$-orbit. The descriptions are as follows. 
 
 \n 
{\it  Suppose that  $G = \mathrm{Aut}(M)$ for  a structure $M$ with no algebraicity. Then \begin{center} $\mathrm{Aut}(G) \cong \mathrm{Aut}(\mathcal{E}_M) = \+ N_G$, and $\Out(G) \cong \+ N_G/G$. \end{center} }
\noindent For a proof see the arXiv version of the paper~\cite[Th.\ 5.6]{arxiv}.     
%\n This will be 	proved as~Theorem~\ref{th:no alg detail}. 

Our main motivation for considering the second class is the following. An  
oligomorphic group is called \emph{$\mathfrak{G}$-finite}    \cite[Definition~3.4.2, pg.~33]{Evans.etal:97}
if each open subgroup contains a least open subgroup of finite index.  (The definition in Lascar's original paper~\cite{Lascar:89} is different, but  coincides with the present one under $G$-compactness.)  We note that   Evans and Hewitt~\cite{evans_profinite} constructed  canonical examples of oligomorphic groups that  are not $\mathfrak{G}$-finite.

%Our main motivation for considering the second class is the following. An 
%oligomorphic group is called \emph{$\mathfrak{G}$-finite} if each open subgroup 
%contains a least open subgroup of finite index. This is the definition used in 
%the context of oligomorphic groups as topological groups 
%(cf.\ \cite[Definition~3.4.2, p.~33]{Evans.etal:97}). Notice that in the 
%literature there is also a different notion of $\mathfrak{G}$-finiteness which 
%goes back to~\cite{Lascar:89}; the two notions are known to be equivalent 
%under $G$-compactness.

 Proposition~\ref{there_are_essential} shows  that every $\mathfrak{G}$-finite group has an essential subgroup. So the class of $\mathfrak{G}$-finite groups that have  only finitely many essential subgroups up to conjugacy is smooth.   It might  be possible a priori   that   \emph{every}  $\mathfrak{G}$-finite group is of this kind.  
 As we will see in Section~\ref{s:wei section}, the  second class   contains the class of  automorphism groups of $\aleph_0$-categorical structures that  have weak elimination of imaginaries, recently studied in~\cite{reconstruction2}.  So \ref{great theorem}(b) above is new in particular for these groups. Each group in this class  is $\mathfrak{G}$-finite   \cite{reconstruction2}.    %Bodirsky et al.\  \cite[Section 3.7]{evans_finite_kang} obtained such a  group that is the automorphism group of an $\aleph_0$-categorical structure  in a finite signature; see the  last paragraph of that paper for a discussion why the group is not  $\mathfrak{G}$-finite,  relying on Lascar~\cite{Lascar:89}.  
   
     \medskip
     
 \n {\it Discussion.}  We do not know whether the isomorphism relation  is  smooth for the   class of all oligomorphic groups, so  we search for \emph{dividing lines}~$D$.   The ideal situation would be  that in the presence of  $D$ one has smoothness, and otherwise not; this is analogous to   what happens with the dividing lines in classification theory. It  would be a strong solution to  the whole problem.  If a Borel equivalence relation is not smooth,  then the relation $E_0$ of almost equality of  subsets of $\omega$ can be Borel reduced to it by a result of Harrington, Kechris and Louveau (see, e.g., \cite[Thm.\  6.3.1]{gao}). We search for the most general condition   ensuring smoothness,  so that we can   focus on  possible constructions for   a reduction of $E_0$   that     exclude groups with that condition.
%	
%We think of Theorem~\ref{thm: main informal 1} as providing two dividing lines. 
%The   application of~(B) we had in mind  is for structures with weak elimination of imaginaries after adding finitely many sorts. However,  venue~(B) and its more general form involving $P$-subgroups are   of independent interest, as it is well possible that all $\mathfrak{G}$-finite groups have only finitely many essential subgroups up to conjugacy. 
Theorem~\ref{great theorem}  suggests  that  the ``right"  dividing line is \emph{$\mathfrak{G}$-finite or without algebraicity}. So groups that fail both conditions would be needed to show the class of oligomorphic groups is not smooth. For this one could  consider   finite covers of non-$\mathfrak{G}$-finite groups, which are shown to exist  in~\cite[Chapter~1]{Evans.etal:97}.  

  \medskip
  
\noindent   \emph{Structure of the paper.}   In Section~\ref{s:Roelcke} we discuss the Polish topology on $\Aut(G)$, for $G$ in the wider class of Roelcke precompact groups $G$. Then we show that $\Inn(G)$ is closed in $\Aut(G)$, thus leading to a Polish topology on   $\Out(G)= \Aut(G)/\Inn(G)$. The core Section~\ref{s:modelty} provides  model theoretic methods needed in the rest of the paper. We use model theory to show that $\+ N_G/G$ is profinite for any oligomorphic~$G$, where $\+ N_G$ is the normaliser of $G$ in $\Sym(\omega)$. Thereafter, we discuss how to represent open cosets of $G$ by pairs of imaginaries, and use this  in   Subsection~\ref{ss: Polish}   to show  that $\Out(G)$ is a t.d.l.c.\ group for oligomorphic $G$.
Section~\ref{ss:PropP} develops the general criterion that allows us  to treat the two rather disparate cases in Theorem~\ref{great theorem}  uniformly; we expect it to have further applications to other Borel classes of oligomorphic groups.   Section~\ref{smooth section} then shows the criterion is applicable to these two classes, and gives further information particular to these classes.

\begin{convention} \label{conv;basic convs} 	 {\rm By $G,H$ we will   denote    closed subgroups of $\Sym(\omega)$.
 All  isomorphisms between topological groups will be topological. $\Aut(G)$ will denote  the group of \emph{continuous} automorphisms of $G$.  We use the group theoretic notation $g^h = hgh^{-1}$ for the conjugate  of $g$ by~$h$.
		By $M$  we always denote a (countably infinite) $\aleph_0$-categorical structure.  By ``definable" we  will mean ``$\emptyset$-definable" unless otherwise noted. }\end{convention}

\section{$\Aut(G)$ for Roelcke precompact $G$}
\label{s:Roelcke}
The main result of this section is that for any Roelcke precompact non-Archimedean group  $G$,   the group of inner automorphisms is closed in $\Aut(G)$. This needs some definitions and preliminaries.

  A  countably based Polish group is  \emph{non-Archimedean} if it has a neighbourhood basis of $1$ consisting of open subgroups. Equivalently, it is topologically isomorphic to a closed subgroup of $\Sym(\omega)$.  Recall  that a \emph{double coset} of a subgroup $U$  of  a group $G$ is a set of the form $U gU$ where $g\in G$. 
Roelcke precompactness~\cite{Dierolf.Roelcke:81} is a general property of topological groups: being  precompact in  the Roelcke uniformity. We will only consider non-Archimedean groups, in which case   the definition can be phrased as follows. 
 \begin{definition}\label{def_precompact} A non-Archimedean group $G$ is \emph{Roelcke precompact} if each open subgroup $U$ has only finitely many double cosets; equivalently, the left action of $U$ on its own \emph{left} cosets has only finitely many orbits.  \end{definition}
 It is well known  that such a group $G$  has only countably many open subgroups: such a subgroup  contains the pointwise stabiliser $V$ of a finite set, and hence is  given as a union of the finitely many double cosets of $V$.  Each oligomorphic group is Roelcke precompact by the following. 
 \begin{fact} [\cite{Tsankov:12}, Theorem 2.4]  \label{fa:RP Olig}A closed subgroup $G$ of $\Sym(\omega)$ is oligomorphic iff $G$ is  Roelcke precompact and the action of $G$ on $\omega$ has only finitely many $1$-orbits.  \end{fact}
Recall that  by $\Aut(G)$ we denote  the group of \emph{continuous} automorphisms of  $G$.
\begin{definition}  \label{df:nbhd} Given a Roelcke precompact group $G$, we define a topology  on  $\Aut(G)$ by declaring  as open the  subgroups of the form \[\{\Phi  \in \Aut(G)\colon \forall {i= 1\ldots n} \  \Phi(A_i)= A_i\},\] where $A_1, \ldots, A_n $ are   cosets of open subgroups of $G$. \end{definition}
\begin{fact}[\cite{Nies.Schlicht:nd}, Thm.\ 1.2, Thm.\ 2.1]  $\Aut(G)$ with the topology  given   by Definition~\ref{df:nbhd}  is  non-Archimedean.   This  topology on $\Aut(G)$ is     the least Baire topology, and the unique Polish topology that makes the action on $G$ continuous.\end{fact}

\cite[Section 3.3]{coarse} develops a    duality between    oligomorphic groups $G$ and   countable structures $\+ M(G)$ called ``coarse groups" defined on the  cosets of open  subgroups of~$G$.  The  {meet groupoid} $\+ W(G)$ is a   variant of   $\+ M(G)$   briefly introduced   in~\cite[Remark\ 2.16]{coarse},   where it is   noted that  $\+ M(G)$ and $\+ W(G)$ are interdefinable via a fixed set of first-order formulas.  The first author and Melnikov~\cite{Melnikov.Nies:22} elaborated on $\+ W(G)$, albeit  in the setting of totally disconnected, locally compact groups $G$  (where its domain consists of the  \emph{compact} open cosets).     It will be    convenient  to use  $ \+ W(G)$  rather than $\+ M(G)$  in what follows.  
\begin{definition} \label{def:meetgr} Given a Roelcke precompact  group $G$, the domain of the \emph{meet groupoid} $\+ W(G)$   consists of the open cosets, together with $\emptyset$. The groupoid product $A \cdot B$ of open cosets is the usual product $AB$ in the group; it  is defined only if   $A$ is a left coset of a subgroup that $B$ is a right coset of, so that  the result is again a coset. The meet is the usual intersection as a binary operation. 
\end{definition}
Note that a   coset  $A$ in $G$ is a right coset of the subgroup  $U= A \cdot A^{-1}$ and a left coset of $V=  A^{-1} \cdot A$. Given that groupoids can be seen as categories, we say that $U$ is the \emph{domain} of $A$ and that $V$ is the \emph{codomain} of $A$. Thus, $A\cdot B$ is defined iff the   codomain of $A$ coincides with the   domain of $B$. By a \emph{two-sided coset} of a subgroup $U$ we mean a left coset of $U$ that is also a right coset of $U$.
The two-sided cosets of $U$ are in a canonical  bijection with $ N_G(U)/U$, whence  their number equals $ [N_G(U) : U]$, the index of $U$ in its normaliser in $G$. 
%; such a coset is also a double coset of $U$. If $gU = Uh$ for some $h \in G$, then since $g \in Uh$ we get $h \in Ug$, hence $Uh = Ug$. Thus the condition reduces to $gU = Ug$, which holds if and only if $gUg^{-1} = U$, i.e., $g \in N_G(U)$. Since two elements $g, g' \in N_G(U)$ define the same two-sided coset if and only if $g^{-1}g' \in U$, the two-sided cosets of $U$ are in bijection with $N_G(U)/U$, and their number equals $[N_G(U) : U]$. 
%\newline \gianluca{Added explanations for the equality, check if OK.}

\begin{fact} \label{fact:conjugates} For   conjugate subgroups  $U,V \in \+ W(G)$,     there is a bijection from the set of two-sided cosets of $U$ to  the set  of   elements of $\+ W(G)$  with domain  $U$ and codomain  $V$. \end{fact}

\begin{proof}
Since $U$ and $V$ are conjugate, there exists a coset $A \in \mathcal{W}(G)$ with domain $U$ and codomain $V$. We claim that the map $\beta$ given by  $\beta(C) =C \cdot A$ is a bijection as required.

\smallskip \noindent \textit{$\beta$  is well-defined: }notice that  $C \cdot A$ is defined  since $C$ has domain $U$ and $A$ has codomain $U$.
Furthermore, $C \cdot A$ has domain $U$   and codomain $V$.  

\smallskip \noindent \textit{$\beta$  is injective:} if $C \cdot A = C' \cdot A$, then multiplying on the right by $A^{-1}$ gives $C \cdot (A \cdot A^{-1}) = C' \cdot (A \cdot A^{-1})$, i.e., $C \cdot U = C' \cdot U$, and since $C$ and $C'$ are already cosets of $U$ this gives $C = C'$.

\smallskip \noindent  \textit{$\beta$  is  surjective:} given any $B \in \mathcal{W}(G)$ with domain $U$ and codomain $V$,  the     product  $B \cdot A^{-1}$ is defined since the codomain of $B$ is $V$, which equals the codomain of $A$, and  hence the domain of $A^{-1}$. The coset $B \cdot A^{-1}$ has domain $U$ (the domain of $B$) and codomain $U$ (the domain of $A$), so it is a two-sided coset of $U$. Furthermore $(B \cdot A^{-1}) \cdot A = B \cdot (A^{-1} \cdot A) = B \cdot V = B$, so $B$ is in the range of $\beta$.
%\newline \gianluca{Add a few details. Please check.} Can't hurt here to calculate a bit with the new notion of MG. just smalll fixes made.
\end{proof}

$\Inn(G)$ denotes  the group of inner automorphisms of $G$. The following implies Theorem~\ref{th:summary part 1}(a). As a hypothesis it would suffice to assume that  each open subgroup  has finite index in its normaliser.
\begin{theorem}  \label{prop: Inn closed}  
		  $\Inn(G)$ is closed in $\Aut(G)$ for each Roelcke precompact group $G$. 
\end{theorem}

\begin{proof}
   %By a   fact  from the  theory of Polish groups (see \cite[Prop.\ 2.2.1]{gao}), it suffices to show that $\Inn(G)$ is $G_\delta$ in $\Aut(G)$. 
%     Let $(A_n)_{n \in \omega}$ be a listing of the open cosets of~$G$ without repetition. Given $\Phi \in \Aut(G)$,  we will define  a set $T_\Phi$ of strings  over some infinite alphabet which is closed under prefixes, and thus can be seen as  a  rooted tree. The alphabet consists of   pairs {of open cosets of $G$. Informally, the $n$-th level of the tree consists of all the pairs of approximations to some element  $g$ and its inverse such that $g A_i g^{-1} = \Phi(A_i)$ for $i< n$. More formally,  we let} 
Let $(A_n)_{n \in \omega}$ be a listing of the open cosets of~$G$ without repetition. Given $\Phi \in \Aut(G)$,  we will define a tree $T_\Phi$ on the alphabet consisting of pairs {of open cosets of $G$}. 
%We define $T_\Phi$ by defining, for every $n < \omega$, the $n$-th level of the tree, denoted as $T^n_\Phi$. $T^n_\Phi$ consists of all the pairs of approximations BRRRRRR 
For this we define $T^n_\Phi$, the $n$-th level of the tree for every $n < \omega$. The $n$-level  consists of all the pairs of approximations to some element  $g$ and its inverse such that $g A_i g^{-1} = \Phi(A_i)$ for $i< n$. Formally, we let:
	\[ T^n_\Phi= \{ \langle (B_i, C_i )\rangle_{i<n} \colon \Phi(A_i)= B_i \cdot A_i \cdot C_i^{-1}   \land \bigcap_{i <  n} (B_i \cap C_i) \neq \emptyset \}.\]
	Note that the domains of   open cosets $B_i$ and $C_i$ as  in this definition are determined by the condition that $\Phi(A_i)= B_i \cdot A_i \cdot C_i^{-1}$.    Also note that $|N_G(U): U| $ is finite for each open subgroup $U$  because each two-sided coset  of $U$ is     a double coset of~$U$. This is implied by our  hypothesis that $G$ is Roelcke precompact  (Definition~\ref{def_precompact}); as mentioned,  in fact the presumably weaker condition is sufficient.
	%(notice that here we are using that $G$ is Roelcke precompact, cf. Definition~\ref{def_precompact}).   Jamming in
	Fact~\ref{fact:conjugates} now shows that  $T_\Phi$ is   finitely branching.

	\smallskip
	
	\n {\bf Claim.} \emph{$\Phi \in \Aut(G)$ is inner if, and only if,   $T_\Phi$ has an  infinite path.}   
	\smallskip
	  
	\n Assuming  the claim, we conclude the argument as follows. By K\"onig's Lemma, $T_\Phi$ has an infinite path iff each of its levels  $T^n_\Phi $   is nonempty. Whether   $T^n_\Phi $ is nonempty only depends on the values   $\Phi(A_0)$, $\ldots$, $\Phi(A_{n-1})$, so the set of such $\Phi$ is clopen in $\Aut(G)$. Thus the condition on $\Phi$  that each level of $T_\Phi$  is nonempty is closed.

	\smallskip \noindent It remains to verify the claim. Suppose $A_i$  is a right coset of a subgroup $U_i$, and a left coset of a subgroup $V_i$. 
	
	For  the implication from \textit{left to right}, suppose that there is $g\in G$ such that    $\Phi(A)= g A g^{-1}$ for each open coset~$A$.  Now let $B_i = g U_i$ and $C_i = gV_i$.  Then $ ( \langle B_i, C_i )\rangle_{i\in \omega}$ is an infinite path on $T_\Phi$: we have  $g \in  \bigcap_{i <  n} (B_i \cap C_i)$ for each $n$, and \
	$B_i \cdot A_i \cdot C_i^{-1}= g U_i A_i V_i g^{-1} = g A_i g^{-1} = \Phi(A_i)$.
%	[ \begin{array}{rcl}	
%		gU_i A_i V_ig^{-1}
%		& = & gU_i h_i V_i g^{-1} \\
%		& = & gA_i V_i g^{-1}\\
%		& = & gk_i V_i g^{-1} \\
%		& = & g A_i g^{-1}.
%	\end{array} \] 
%	not first course in group theory

	\smallskip \noindent For  the implication from \textit{right  to left}, let $f $ be an infinite path on $T_\Phi$, and write $\langle B_i, C_i \rangle = f(i)  $.
	For each $n$, choose $g_n \in  \bigcap_{i <  n} (B_i \cap C_i)$. We will show that $\lim_n g_n$ and $\lim_n (g_n^{-1})$ exist. Given $m$,  let  $r=r(m) $ be the number such that
	\[A_{r(m)} = G_{(0, \ldots, m-1)},\] 
	where $G_{(0, \ldots, m-1)}$ denotes the pointwise stabilizer of $\{0, ..., m-1\}$ in $G$. Note that  $B_r$ and $C_r$ are    left cosets of $  G_{(0, \ldots, m-1)}$ since $B_r \cdot A_r$ and $A_r \cdot C_r^{-1}$ are  defined. If $n \geq r$ then for each $i< m$, we have $g_n(i) = g_{r}(i) $, because $g_n^{-1} g_{r} \in B_r^{-1}  \cdot B_r    = G_{(0, \ldots, m-1)}$, Similarly,  $g^{-1}_n(i) = g^{-1}_{r}(i) $ because $g_{r} g_n^{-1}\in C_r^{-1} \cdot C_r=  G_{(0, \ldots, m-1)}$.
Thus   $g(i):= \lim_n g_n(i)$ and  $h(i):= \lim_n g^{-1}_n(i)$ exist  for each $i$, as required.

Clearly the functions $g \colon \omega \to \omega$ and $h \colon \omega \to \omega$ are inverses of each other. Thus  $g \in G$, as $g$ is  a permutation that is a pointwise limit of elements of $G$.

For each $i \in \omega $ we have $g = \lim_{n> i } g_n  \in B_i \cap C_i$ because $B_i \cap C_i$ is closed.  	Thus  $B_i= g U_i$ and $C_i = gV_i$, and  we have  for each~$i$ \begin{center} $g A_i  g^{-1}= g (U_i \cdot A_i  \cdot V_i) g^{-1} =  B_i \cdot A_i \cdot C^{-1}_{i}  = \Phi(A_i)$,  \end{center}  so $\Phi$ is inner.
	\end{proof}
	
%	\smallskip \noindent Given $\ell \in \omega$,  we verify that  $g \in B_\ell$.
%	There is $m$ with $r(m) > \ell$ such that $G_{(0, \ldots, m-1)}$ is contained in the domain $U$  of $B_\ell$ (in the sense of categories as discussed before  Fact~\ref{fact:conjugates}). 
%	We have   $g(i)  =g_{r(m)}(i)  $ for each $i< m$. Thus $g^{-1} g_{r(m)}\in G_{(0, \ldots, m-1)} \sub U$. Since $g_{r(m)} \in B_\ell$, this implies that $B_\ell  = g_{r(m)}U = gU $, so $g \in B_{\ell}$ as required.   Similarly, there is  $m$ with $r(m) > \ell$ such that $G_{(0, \ldots, m-1)}$ is contained in the codomain of $C_\ell$, which equals the   domain of $C^{-1}_\ell$. Thus,  arguing as before, $g \in C_\ell$ for each $\ell \in \omega$. 

\begin{remark}  \label{rem:top Inn} The topology on $\Inn(G)$ inherited from $\Aut(G)$ is the expected one. Indeed, the centre $Z=Z(G)$ of   $G$ is   closed. We have a Polish   topology on 
	$ G /Z$ by declaring $U Z /Z$ open iff $UZ$ is open in $G$. %, note that  $G/Z$ into a non-Achimedean group. 
	The canonical  isomorphism $L \colon G/Z \to \Inn(G)$ is continuous. Since $\Inn(G)$ is Polish as a closed subgroup of $\Aut(G)$,  $L$ is a homeomorphism by a standard result in the theory of Polish groups (see~\cite[2.3.4]{gao}). \end{remark}

We note that a hypothesis such as Roelcke precompactness is needed in~\ref{prop: Inn closed}. Wu~\cite{Wu:71}  gave an example of a discrete  group $L$ such that $\Inn(L)$ is not closed in $\Aut(L)$. A slight modification yields a nilpotent-2 group of exponent 3.   (This  group resembles   the groups  recently studied in \cite{Nies.Stephan:24}; in particular, it is finite automaton  presentable.)
 
 \begin{example}[Similar to Example 4.5 in~\cite{Wu:71}] {\rm Let $G$ be generated by  elements    $a_i,b_i, c$ order $3$ ($i \in \mathbb N)$, where $c$ is central, with the relations, for $i \neq k$  \begin{center} $b_ia_i b_i^{-1}= a_ic $, $[a_i,a_k] = [a_i,b_k]= [b_i,b_k]= 1$. \end{center}  
 Let $\Phi$ be the automorphism  of $L$ given by $\Phi(a_i)= a_i c$, $\Phi(b_i)= b_i$, and $\Phi(c) =c$. It  is not inner, but in the closure of $\Inn(G)$.}
 \end{example}
 \begin{proof} To  check that $\Phi$ is indeed in $\Aut(G)$, note that  $c$ can be omitted from the list of generators. Given a   word $w$ in $a_i,b_i$  where each letter occurs with exponent $1$ or $2$,  one uses that  $\Phi(w)= wc^{m \mod 3}$, where $m$ is the number of occurrences of   $a_i$'s. The inverse of $\Phi$ is given by $a_i \mapsto a_i c^{-1}$ and  as before $b_i \mapsto b_i$, $c \mapsto c$. 
 
 Write  $g_n = \prod_{i< n} b_i$. We have  $g_n a_i g_n^{-1}= a_ic$ for each $i< n$, and $g_n b_ig_n^{-1} = b_i$ for each $i$. Letting $\Phi_n$ be conjugation by $g_n$, we have $\lim_n \Phi_n = \Phi$ in $\Aut(G)$. 
 	For each $g\in G$, conjugation by $g$ fixes almost all the generators. So $\Phi$ is not inner. 
 \end{proof}

\section{Some relevant model-theoretic concepts with  applications} \label{s:modelty}
The language we use in this paper  is mainly   from the areas of  topological groups and permutation groups.   The following   well-known fact    can be used as a mechanism to translate from   group theoretic language  to model theoretic language and back. The notation  $M_G$  below should not be confused with the notation of a coarse group $\mathcal{M}(G)$ from~\cite{coarse}.

\begin{fact}[see \cite{hodges}, Theorem 4.1.4] \label{df: MG} Let $G$ be oligomorphic. We have    $G= \Aut(M_G)$ for the $\aleph_0$-categorical relational structure $M_G$ that has an $n$-ary relation $P^n_i$ for each $n$-orbit of $G$. In this context one says that $M_G$ is a \emph{canonical structure} for $G$; it is unique up to  permuting the names of  the $k$-orbits, for each $k$.  
\end{fact}   
The operation $G \mapsto M_G$ is easily seen to be Borel~\cite[Section 1.3]{coarse}. Recall that  for oligomorphic groups  $G,H$, we have that $G \cong H$ iff $M_G$ and $M_H$ are bi-interpretable (for a definition of bi-interpretability see \cite[Chapter~5]{hodges}).  The  model theoretical language  is of great importance in  proving  results in  the setting of oligomorphic groups.  For instance,  it is  not hard to see that bi-interpretability of $\aleph_0$-categorical structures is Borel~\cite[Section 1.3]{coarse}; so  topological isomorphism of oligomorphic groups  is also Borel  (where  naively it looks merely analytical).  Furthermore, using that  the operation $G \mapsto M_G$ is Borel, one can   verify that classes of groups such as the ones having no algebraicity (Definition~\ref{no alg}) are Borel by looking at the corresponding classes of structures. 

This section    uses model theoretic language to provide results that will be important later on. For instance, we get a grip on  the open cosets of $G$ by representing such a coset as a pair of imaginaries of the same sort.
% up to a certain form of interdefinability.  

\subsection{Imaginaries and $M^{\mrm{eq}}$}
  %An imaginary over  a structure $M$ is a class of a definably equivalence relation on~$M$.
  The structure  $M^{\mrm{eq}}$    of  {imaginaries} over a structure  $M$  goes back  to Shelah~\cite{classification}.We  follow the treatment in Hodges~\cite[p.\ 151]{hodges} (see also \cite[pg.\ 12-13]{evans}).
Recall that `definable' means $\emptyset$-definable by Convention~\ref{conv;basic convs}.  The structure $M^{\mrm{eq}}$ has sorts $S= D/E$,  where $D \subseteq M^k$ (for $1 \leq k < \omega$) is definable and $E$ is a definable equivalence relation on $D$; its elements are written as $\ol a /E$ for $\ol a \in D$. 
  %We will use $\aaa, \beta, \ldots$ for elements of $M^{\mrm{eq}}$. 
   The atomic   relations of $M^{\mrm{eq}}$  are the  atomic  relations of the given structure $M$ (viewed as relations on the  ``home sort"~$M$), as well as the graphs of the projections $D \to S=D/E$ for each sort $S$. 
   
   For sorts $S_i = D_i/E_i$, $i = 1, \ldots, n$, we have the product sort $\prod S_i  = \prod D_i / \prod E_i$,  where the equivalence relation $ \prod E_i$ is defined component-wise. 
 
 In a many-sorted structure, each definable relation is a subset of $\prod_{i< n} S_i$ for sorts  $S_0, \ldots, S_{n-1}$.  The   definable relations are given  as usual via   formulas in the   first-order logic in the corresponding signature  with variables for each sort.  Extending the usual  definition of algebraic/definable closure to a many-sorted structure is no problem. We note that when  $M$ is $\aleph_0$-categorical, $\text{acl}(A)  \cap S$ is finite for each finite $A \sub M^{\mrm{eq}}$ and  each sort $S$.

\subsection{Orbital structures and profinite groups}
%Recall  that   $G= \Aut(M_G)$ where $M_G$ is  the $\aleph_0$-categorical relational structure that has a $k$-ary relation for each $k$-orbit of $G$. 
 A \emph{reduct} of a structure $M$ is a structure $N$ with the same domain such that each relation or function of $N$ is definable in $M$. (In some references the notion of ``reduct'' is limited to restrictions of the language; the two notions can be reconciled  using  Morleyzation in the sense of  \cite[pg.~63]{hodges}.) %c.f. e.g.- aaahrg 
 
   The reduct relation is a  quasi-order among structures with a given domain;  the corresponding equivalence relation is    interdefinabilty, namely,  that two structures have the same definable relations.  It is well known that for $\aleph_0$- categorical structures $N,M$, $N$ is  a reduct  if $M$ iff $\Aut(M)$ is a subgroup of $\Aut(N)$. 
	\begin{definition}[Orbital structure]\label{def_EM} Consider the  language $\{\sim_n : 0 < n < \omega \}$ where for each $n$, the relation symbol $\sim_n$ is  of  arity $2n$. Given an $\aleph_0$-categorical structure $M$ with domain $\omega$, by $\+ E_M$  we denote the structure in this language such that  $\bar{a} \sim_n \bar{b}$ iff there is $g \in \mrm{Aut}(M)$ such that $g(\bar{a}) = \bar{b}$. One says that  $\+ E_M$ is the \emph{orbital structure} of $M$.
\end{definition} 
	It is easy to check  that the structure $\+ E_M$ from Definition~\ref{def_EM} is a reduct of $M$.
	
%	, to see this, let $n < \omega$ and let $\Delta_1, ... \Delta_{k_n}$ be the $\mrm{Aut}(M)$-orbits on $M^n$, then we have that  $\bar{a} \sim_n \bar{b}$ if and only if, for all $i \in [1, k_n]$,  $\bar{a} \in \Delta_i \leftrightarrow \bar{b} \in \Delta_i$, and so recalling that the $\Delta_i$'s are $\emptyset$-definable sets we are done. 
%	
%	
%	\gianluca{Added this. MAYSF.}
%\andre{everything in one sentence, not good.}

By $\+ N_G$ we denote the normaliser of $G$ in $\Sym(\omega)$.   By the following,      $\Aut(\+ E_M)$ is  the normaliser   of $G= \Aut(M)$ in $\Sym(M)$, the group of permutations of~$M$.
\begin{fact} \label{fa:NG} Let $M$ be an $\aleph_0$-categorical structure with domain $\omega$. Let $G = \Aut(M)$. 
 \begin{itemize} \item[ (i)] $\+ N_G= \Aut(\+ E_M)$.

 \item[(ii)] $\+ E_M$ is up to interdefinability the smallest reduct $L$ of $M$ such that $\mrm{Aut}(M) $ is a normal subgroup of $ \mrm{Aut}(L)$. \end{itemize}
\end{fact}

 %(recall that the reduct relation is a partial order, as observed right before Definition~\ref{def_EM}).
 
\begin{proof} (i) To show that $\+ N_G\le  \Aut(\+ E_M)$, let $h \in \+ N_G$. 
 %(see  \ref{def_EM}):  
 Suppose that $g \cdot \ol a = \ol b$ where $g \in   G$ and $\ol a,\ol b\in M^n$.  Let $g' = g^h \in G$. Then $g' h \cdot \ol a = h \cdot \ol b$.  So  $h $ preserves the relations~$\sim_n$.

 For the converse inclusion, we may assume that $M= M_G$. If $h \in \Aut(\+ E_M)$ and $g$ preserves each relation $P^n_i$, then so does $g^h$. Thus $h \in \+ N_G$.  
 
 \n (ii) Let $R$ be a reduct of $M$. By definition of a normaliser, $G$  is normal in $\Aut(R)$  iff $\Aut(R) \le \+ N_G$. By  (i) this is  equivalent to  $\Aut(R) $ being contained in  $\Aut(\+ E_M)$, which in turn is equivalent to   $\+ E_M$ being  a reduct of $R$.
  \end{proof}

	\begin{thm}\label{profinite_lemma} Let $M$ be an $\aleph_0$-categorical structure. Then $V=\mrm{Aut}(\mathcal{E}_M)/\mrm{Aut}(M)$ is a profinite group. If $M$ is interdefinable with a structure in a   finite signature, then $V$ is finite.
\end{thm}

	\begin{proof} Since $M$ is interdefinable with $M_{\Aut(M)}$,   we may assume that   for every $n$ with $0 < n < \omega$, the structure $M$ has   predicates $P^n_1, ..., P^n_{k(n)}$ for each of its $n$-orbits, and no other relations or functions. For   $n \leq \omega$ let $M_n$ be the structure obtained after removing the predicates of arity greater than $n$, and let $\mathcal E_n:= \mathcal E_{M_n}$ be the corresponding orbital structure as in Definition~\ref{def_EM}. 
	
\smallskip \noindent	 For $k \leq m \leq \omega$ the identity map  $\mathcal E_{m} \to \mathcal E_k $  induces a map $$q_{m,k} \colon \Aut(\+ E_m )/ \Aut (M_m) \to \Aut(\+ E_k )/\Aut(M_k).$$  Writing $p_n= q_{n+1,n}$, this yields an inverse system %\andre{don't put colon !!!! }
	\begin{equation*} \tag{$*$}
(\mrm{Aut} (\+ E_n)/\Aut(M_n) , p_n)_{n \in \mathbb N}.  
	\end{equation*}
	The group $\mrm{Aut}(\mathcal{E}_n)/\mrm{Aut}(M_n)$ is canonically isomorphic to a subgroup of $S_{k(n)}$  (the group of permutations of $\{1, \ldots, k(n)\}$): an isomorphism is induced by  mapping  $\pi \in \mrm{Aut}(\mathcal{E}_n)$ to the permutation $\alpha\in S_{k(n)}$ such that $\pi(P^n_i)= P^n_{\alpha(i)}$ for each $i \leq k(n)$. Let 
  $$  L: =  \varprojlim_n (\mrm{Aut} (\+ E_n)/\Aut(M_n) , p_n)$$
  be the inverse limit in the category of topological groups, where the finite groups $\mrm{Aut}(\mathcal{E}_n)/\mrm{Aut}(M_n)$ carry the discrete topology.  Every $f \in L $ can be concretely seen as a certain sequence of permutations $(\alpha_n)_{0< n < \omega}$ where $\alpha_n\in S_{k(n)}$. 
 \emph{We claim  that $V \cong L$ via the continuous homomorphism  $F\colon V\to L$ induced by the   maps $q_{\omega, n} : V \to \Aut(\+ E_n )/\Aut(M_n)$}. To verify this,  we   invoke the universal property of the inverse limit. To see that   $F$ is 1-1, suppose $\pi \in \Aut(\+ E_\omega) \setminus \Aut(M)$. Then there is $n$ and $i \leq k(n)$  such that $\pi(P^n_i)= P^n_j$ with $i \neq j$. So the image of $\pi \Aut(M)$ in  $\mrm{Aut} (\+ E_n)/\Aut(M_n)$ is not equal to the identity. Thus $F(\pi \Aut(M))$ is not equal to the identity. 
	
\smallskip \noindent		To see that $F$ is onto,  given $f\in L$,  let $(\alpha_n)_{0< n < \omega}$ be as above, and  let $T_f$ be the theory in the language of $M$ extended by  an additional function symbol $\pi$, and with the following set of  axioms: 
  
 \begin{enumerate}[(1)] \item  $\mrm{Th}(M)$;
\item ``$\pi$ is a permutation"; 
  
\item  ``$\pi(P^n_i)= P^n_{\alpha_n(i)}$,   for each $n$ and each $i \leq k(n)$".\end{enumerate}
Note that $T_f$ is consistent  by the Compactness Theorem, using  the assumption that each $\alpha_n$ is induced by an automorphism $\pi_n$ of $\+ E_n$. Let $N$  be  any countable model   of $T_f$, and let $N'$ be the reduct of $N$ to the language of $M$. Since $M$ is $\aleph_0$-categorical, there is an isomorphism $h \colon M \to N'$. We have  $g=h^{-1} \circ \pi^N \circ h \in \Aut(\+ E_M) $ and $F( g \Aut(M))= f$. As $f\in L$  was arbitrary this shows that $F$ is onto. 

\smallskip \noindent The group isomorphism $F^{-1}\colon L\to V$ is clearly Borel, so it  is   continuous as a Borel (and hence Baire measurable) homomorphism between Polish groups; see, e.g., \cite[Th.\ 2.3.3]{gao}.

\smallskip \noindent Suppose  the structure $M$ is interdefinable with a structure in  a finite signature, and let $n$ be the maximum arity of a symbol in the signature. Then for each $r \ge n$, the structure $M_r$ is an extension by first-order definitions of $M_n$, and hence $\Aut(M_r)= \Aut(M_n)$, and  $\Aut(\mathcal E_r)= \Aut(\mathcal E_n)$. Hence   the inverse system 
$(*)$ above  is eventually constant,   whence its  inverse limit is finite.
\end{proof}

%\begin{remark} The profinite group  $L$ can be  described purely in terms of the theory of $M$.  From $T= \mrm{Th}(M)$ we build a structure $\+D_T$ with infinitely many sorts $D_n$ such that $|D_n| = k(n)$ is the number of $n$-types of $T$.  For each $n$ and each $i \le n$ there are functions $f^n_i$	
%	we can obtain the ternary  function $\theta$  on $\omega$      such that, for each positive $n$,   $r \leq k(n)$ and $i \leq n+1$,   the relation obtained by erasing the $i$-th component of the tuples in the $(n+1)$-orbit $P^{n+1}_r$ equals $P^n_{\theta(n,r,i)}$. Then   $L$  consists of the sequences of permutations  $(\alpha_n)_{0< n < \omega}$ that cohere with taking these projections in the sense  that, for each $n,r,i$ as above, we have \[\aaa_n(\theta(n,r,i))= \theta (n,\aaa_{n+1}(r),i).\]
%For  let $N$ be the model such that $(P^n_i)^N= (P^n_{\aaa_n(i)})^M$ for each $n$ and $i \le k(n)$. Then $N \models \mrm{Th}(M)$ because the theory has quantifier elimination. Since the theory is $\aleph_0$-categorical, there is an isomorphism $\pi \colon M \to N$. Then $\pi \in \Aut (\+ E_M)$ and $\pi \Aut(M)$ is mapped to the sequence  $(\alpha_n)_{0< n < \omega}$. 
%\end{remark}
% 
   
	\begin{cor}\label{profinite_lemma+} Let $M$ be $\aleph_0$-categorical. Let $L$ be a reduct of $M$ such that $\mrm{Aut}(M) $ is a normal subgroup of $ \mrm{Aut}(L)$. Then $\mrm{Aut}(L)/\mrm{Aut}(M)$ is a profinite group. 
\end{cor} 
% 
% \andre{makes not sense  in general, e.g. when N=  trivial reduct, because  Aut(M) is not normal in $\Sym(\omega)$. 
% With  normality as a hypothesis it works and this is all we need!}
% 
 \begin{proof}  $\mrm{Aut}(L)$ is contained  $\Aut(\+ E_M)$ by Fact~\ref{fa:NG}. 
Since $\Aut(L)$ is also closed in $\Sym(\omega)$,  we conclude that  $\mrm{Aut}(L)/\mrm{Aut}(M)$ is profinite as a closed subgroup of $\mrm{Aut}(\+ E_M)/\mrm{Aut}(M)$. \end{proof}

%\gianluca{Would you agree with the statement that $\+ E_M$ is the smallest reduct $R$ of $M$ such that $\mrm{Aut}(M) $ is a normal subgroup of $ \mrm{Aut}(R)$? Isn't this what we are proving? If so a sentence like the one I suggest seems sub-basic to me.} \andre{you mean smallest in terms of the definability preordering? Yes this follows from Fact 2.3, should I put it? Might be known to people like Evans. We should definitely get his comments on the paper, also Pinsker/Bodirsky and perhaps Schlicht}

%\andre{If you divide a top group $G$ by a closed normal subgroup $N$, the projection map $p$ is open.  So the image of a closed subgroup $C$ of $G$ that contains $N$ is closed in $G/N$, as $p(G/C) = G/N \setminus p(C)$.}

% 
%	\begin{fact}\label{no_alg} If $M$ is $\aleph_0$-categorical with no algebraicity, then $\mrm{Aut}(M)$ has trivial center. 
%\end{fact}

\subsection{Open cosets and pairs of imaginaries}  
\label{ss:open coset}  
 In    this section, $G$ will be an oligomorphic closed subgroup of $\Sym(\omega)$.  Recall that  $G= \Aut(M)$ where $M=M_G$ is the canonical structure for $G$. Note that  $G$ acts naturally on $M^{\mrm{eq}}$.   Here we describe  open cosets and the inclusion relation between them model-theoretically via $M^{\mrm{eq}}$.  We begin with open subgroups.
 The following technical notion will  be needed in Section~\ref{s: no alg}. 
  \begin{definition}[\cite{rubin}] \label{df:degen} Let $M$ be a set, $D \subseteq M^n$ for $n \geq 1$,  and let $E$ be an equivalence relation on $D$. The relation $E$ is called \emph{degenerate} if for some   set $\sigma \subsetneq n$, the equivalence relation $E_{\sss,D}= \{\langle  {a},  {b}\rangle  \in D^2 \colon \forall i \in \sigma \, [   a_i=  b_i]\}$ is contained in $E$. Otherwise $E$ is called \emph{non-degenerate}.
\end{definition}

For $\alpha \in M^{\mrm{eq}}$, let $G_\alpha = \{ g \in G \colon g(\alpha) = \alpha \}$.   
 The  open subgroups coincide with the groups of the form $G_\aaa$ by the following  well-known fact that is  implicit  in the proof of~\cite[Theorem 1.2]{biinte}.  The present form is close to~{\cite[Theorem~2.7]{evans}}.
  %See Theorem 2.7 in   \href{http://wwwf.imperial.ac.uk/~dmevans/Bonn2013_DE.pdf}{ David Evans' 2013  notes}.
 % $U \le_o G$ denotes that $U$ is an open subgroup of $G$. 

  \begin{lemma}\label{evans_fact} \label{fac1} Let $H$ be an open subgroup of $G$.    Suppose that     $G_{\bar{a}} \leq H$ where $\bar{a} \in M^n$  (note that  such an  $\bar a$ exists as $H$ is open). 
\begin{enumerate}[(i)]
\item There are  definable relations $D \subseteq M^n$ and  $E \sub M^{2n}$ such that  $E$ induces an equivalence relation  on $D$, and  $H = G_{\bar{a}/E}$ for some $\bar a \in D$. 

\item  For every $\bar{a} \in M^n$, only finitely many open subgroups $H$ of $G$   contain $G_{\bar{a}}$.

\item  If  a tuple $\bar a$ is chosen of minimal length such that $G_{\bar a} \leq H$, then    $E$, seen as an equivalence relation on $D$, is  non-degenerate. \end{enumerate} \end{lemma}

\begin{proof} (i)   Let $D = G\cdot \bar a$, and let the  equivalence relation $E$ on $D$ be given by 
  \[  h \bar a E g  \bar a  \leftrightarrow g^{-1}h \in H.\]
  Since $D$, as well as   $E$ viewed as a $2n$-ary relation,  are invariant under the action of $G$, they are   both  definable. Let $\alpha = \bar a/E$. 
  Then we have $G_\alpha = H$, because $g(\alpha) = \alpha $ iff $  g(\bar a ) E \bar a$ iff $g \in H$.
  
    \n (ii)  is   immediate from (i)  because $M$ has  only finitely many $2n$-orbits.

\n   (iii) Assume for a contradiction that  $E$ is degenerate via $\sigma$. We may assume that $\sigma=\{0, \ldots, n-2\}$. Let $\ol a' $ be $ \ol a$ with the last entry removed.  We claim that $G_{\ol a'} \leq H$. For let $h\in G_{\ol a'}$, and let $\ol b$ be the tuple $\ol a'$ with $h(a_{n-1})$ appended, so that  $\ol b = h(\ol a)$.  
Then \mbox{$\ol b \,  E_{\sigma, D} \, \ol a$} and hence $\ol b \, E \, \ol a$ hold,     so  $h \in H$ by the definition of $E$.
\end{proof}
 
  The containment relation between open  subgroups is described by relative definability in $M^{\mrm{eq}}$.   
 
  \begin{fact}  \label{fac2} For $\aaa, \beta \in M^{\mrm{eq}}$, we have 
 $G_\aaa \sub G_\beta $ $\Leftrightarrow$  $ \beta \in \text{dcl}(\aaa)$.  
 \end{fact}
 \begin{proof}
 The implication ``$\LA$'' is immediate.

\smallskip
\noindent For the implication ``$\RA$", suppose $\aaa = \ol a /E$ and $\beta = \widetilde b /F$ as above. Then 
 \begin{center}  $h \ol a  E g \ol a  \lra g^{-1}h \in G_\alpha$  and   $h \widetilde b  F g \widetilde b  \lra g^{-1}h \in G_\beta$.  \end{center}
 Let $R=  G \cdot \langle \ol a , \widetilde b\rangle$, and note that $R$ is a left invariant, and hence  definable, relation. 
 Since $G_\alpha \sub G_\beta$, 
 %and the orbit equivalence relation of $G_{\alpha}$ is definable from $\alpha$, 
 $\beta$ is  the unique equivalence class of $F$ containing $\widetilde  b '$ for some $\langle  \bar a',\widetilde b'\rangle \in R$ with $\bar a '\in \alpha$. 
Hence $\beta \in \text{dcl}(\aaa)$. 
% 
% (ii) If $U \sub V$ where $U = G_\aaa$, $\aaa = \ol a/E$,   by \cref{fac1}  we have  $V = G_\beta $ for some imaginary $\beta = \ol a /F$, and $E \sub F$. So the number of open subgroups containing $U$ is bounded by the number of definable equivalence relations on $G\ol a$ containing $E$.  
 \end{proof}

 %\noindent {\it Open cosets.}
 
 Next we represent open cosets by    pairs of imaginaries in the same orbit of $M^{\mrm{eq}}$.
 
 \begin{notation} \label{not:cosets} For imaginaries $\aaa_0, \aaa_1$ of the same sort of $M^{\mrm{eq}}$, let
 \[ [\aaa_0, \aaa_1] =  \{ g \in G \colon g(\aaa_1)= \aaa_0\}. \] \end{notation}
%\gianluca{As this notation is non-standard in my opinion it should be made into a "Notation" environment and refereed to later, e.g. in S2.5 you should refer to it.}
  If $ [\aaa_0 , \aaa_1] \neq \emptyset$, then  clearly $ [\aaa_0 , \aaa_1]$ is a right coset of $ G_{\aaa_0}$,  and    a left coset of $ G_{\aaa_1}$. Also $ [\aaa_1 , \aaa_0]=  [\aaa_0 , \aaa_1]^{-1}$.  We note that  for each  sort $S$, the  relation  $\{\langle \aaa_0, \aaa_1 \rangle  \in S^2 \colon [\aaa_0 , \aaa_1] \neq \emptyset \}$ is definable:  it says that $\aaa_0$ and $\aaa_1$ are in the same orbit under  the action of $G$ on  $M^{\mrm{eq}}$; now use that  each orbit on  $M^{\mrm{eq}}$ is definable. 

  \begin{fact} \label{coset and pairs} Each open coset $A$ is of the form $[\aaa_0, \aaa_1]$. Here $\aaa_1$ can be chosen to be any element of  $M^{\mrm{eq}}$ such that $A$ is a left coset of $G_{\aaa_1}$ (which exists by~\ref{evans_fact}(i)).
 %, where $A$ is a right coset of $G_{\aaa_0}$, and a left coset of $G_{\aaa_1}$.
  \end{fact}
 
 \begin{proof}  Note that   $A = h G_{\aaa_1}$ for any  $h \in A$; now  let $\aaa_0 = h(\aaa_1)$. \end{proof}

Inclusion of open cosets, and the product of cosets in case it is  defined in $\+ W(G)$,  can also be expressed in terms of definability; see \cite[4.8]{LogicBlog:23}.

 \begin{example} {\rm 
 Let $k \geq 2$. Let $M$ be the countably infinite structure with one permutation     $\pi$, where  each cycle   has length $k$. Then $G:= \Aut(M)\cong  C_k \wr \Sym(\omega)$ where $C_k$ is a cyclic  group of size $k$, and  the wreath product is taken via the action of $\Sym(\omega)$ on~$\omega$. Note that $G$  acts 1-transitively on $M$.   Let $V= G_b$ where $b\in M$. Each subgroup $U$ of the form $ G_a$  has $k$ right cosets $A_i$ $(i< k)$ that are left cosets of $V$: these are the   $A_i = \{ f\in G \colon  \pi^{i}(a)=  f(b)\}$ where $0\leq i < k$. Letting $a= b$ we see that  the group  $N_G(V)/V$ is cyclic of size $k$.  
 
 In general, if $G= \Aut(M)$ is 1-transitive and   $V = G_b$ for some $b \in M$, then $N_G(V)/V$     is isomorphic to the centraliser of $G$ in $\Sym(\omega)$ via the map $\pi \to [\pi(b), b]$. The centraliser is non-abelian in a variant of the above example: replace the cyclic group above by a subgroup of $S_n$ with non-abelian centraliser. }
 \end{example}
% blocks of size 6, in each permutation f is (01)  and g is (024)(135). Then get C_2\wr C_3.

  \subsection{$\Out(G)$ is t.d.l.c.\ for oligomorphic groups $G$}  \label{ss: Polish} 
  Recall that $\Inn(G)$ is the group of inner automorphisms of $G$, which by \ref{prop: Inn closed}  is a closed subgroup of $\Aut(G)$ whenever $G$ is Roelcke precompact.
 We now obtain  Theorem~\ref{th:summary part 1}(c).
  \begin{theorem}  \label{prop: tdlc} Let $G$ be an oligomorphic group,  and   $\+ N_G$   its normaliser in $\Sym(\omega)$. \begin{enumerate}[(i)]
\item   Let $S \colon \+ N_G \to \Aut(G)$ be the homomorphism defined by  $S(\pi)(g) = \pi \circ g \circ \pi^{-1}$.  \\
    The subgroup   $\mrm{range} (S)$ is open in $\Aut(G)$. 
     \item  $\Out(G)$ is  totally disconnected and   locally compact, having the profinite group  $\+ P=\mrm{range} (S)/\Inn(G)$ as an open subgroup.  If $M_G$ is interdefinable with a structure in a   finite signature, then $\+ P$ is finite, and hence $\Out(G)$ is discrete.
\end{enumerate}
  \end{theorem}

 \begin{proof}
 {\bf  (i)}.  We will use Notation~\ref{not:cosets}, where $M$ is the canonical structure for $G$ as in~\ref{df: MG} (with domain $\omega$). Note  that for $a,c,d \in \omega$,  if $[d,a]\supseteq [c,a] \neq \emptyset$ then $d=c$.
% :  By \ref{fac3}, there is a definable function $\eta $ on $\omega$ such that $b= \eta(c)$ and $a = \eta(a)$. Such a function commutes which each $g\in G$. Since $c = g( a) $ for some $g \in G$, we have $b = \eta ( g ( a)) = g (\eta (a)) = c$. 

\smallskip 
\noindent Let  $a_1, \ldots, a_n$ represent the $1$-orbits of $G$. It suffices to show that the open subgroup $\{ \Phi \in \Aut(G) \colon \, \Phi(G_{a_i})= G_{a_i} \text{ for each } i\}$ is contained in $\mrm{range} (S)$.  
 Suppose that   $\Phi $ is in this subgroup.  Since $\Phi$ fixes each subgroup $G_{a_i}$, for each coset $D=[b,a_i]$, $\Phi(D)$ is also a left coset of $G_{a_i}$. Hence, by~\ref{coset and pairs}, $\Phi(D)$  can be written in the form $[d,a_i]$, and $d$ is unique as noted above. Define a function $\pi_\Phi  \colon \omega \to \omega$  by 
 \[ \pi_\Phi(b)= d \text { if } \Phi([b,a_i]) = [d,a_i], \text{ where } [b,a_i ] \neq \emptyset\]
 (noting  that $i$ is uniquely determined by $b$). %
 Clearly  $\pi_\Phi$ and $\pi_{\Phi^{-1}}$ are inverses. So $\pi_\Phi\in \Sym (\omega)$.    The following establishes (i).
 
 \smallskip
   \n {\bf Claim.} \emph{  $\pi_\Phi \in \+ N_G$ and $R(\pi_\Phi) = \Phi$. } \newline 
    \smallskip
 Write $\pi=\pi_\Phi$. To verify the claim, we first show that   $\Phi([r,s]) = [\pi(r),\pi(s)]$  for each $r,s \in \omega$ such that $[r,s]\neq \emptyset$. 
Note that by hypothesis on $\Phi$ and since $G_{a_i} = [a_i,a_i]$, we have  $\pi(a_i)= a_i$ for each $i$. Also note that $\Phi([b,c])^{-1}= \Phi([c,b])$  for each $b,c \in \omega$. 
 Let now $i$ be such that $r$ and $s$ are in the $1$-orbit of $a_i$. We have $[r,s] = [r,a_i]\cdot [a_i,s]= [r,a_i]\cdot [s,a_i]^{-1}$. So 
 \[\Phi([r,s])= \Phi([r,a_i])\cdot \Phi([s,a_i])^{-1}= [\pi(r), a_i]\cdot [\pi(s), a_i]^{-1} = [\pi(r), \pi(s)]. \]

\smallskip \noindent Next,  for each $h \in G$ we have $\{ h\}= \bigcap_{s\in \omega}  hG_s$, and $hG_s = [h(s), s]$ using Notation~\ref{not:cosets}.   Then 
 
 \begin{center}  $ \{\Phi(g)\} = \bigcap_{s\in \omega} \Phi( [g(s), s])= \bigcap_{s\in \omega} [\pi(g(s)), \pi (s)] = \bigcap_{t\in \omega}  [\pi(g(\pi^{-1}(t))), t]= \{ g^\pi\}$,\end{center}
 as required.  This shows the claim. 
  
\smallskip 
\noindent  {\bf  (ii)}.   The kernel of the homomorphism $S$ is the centraliser $C= C_{\Sym(\omega)}(G)$ of $G$ in $\Sym(\omega)$, which consists of the  permutations that are definable in $M$ when seen as binary relations on $M$. Note that  $C$ is finite because $M$ is    $\omega$-categorical. By general group theory $C$ is a  normal subgroup of $\+ N_G$, hence  $GC$ is also  a normal subgroup of $\+ N_G$.    So $G$ has finite index in $GC$, whence  $GC$ is closed in $\Sym(\omega)$. Since $C$ is the kernel of $R$ and $S(G)= \Inn(G)$, we have $S^{-1}(\Inn(G)) = G C$. Thus   $\mrm{range}(S)/\Inn(G)$ is  topologically  isomorphic to $\+ N_G/GC$.     By Fact~\ref{fa:NG} $\+ N_G= \Aut(\+ E_M)$, so by the first statement of  Theorem~\ref{profinite_lemma} the group $\+ N_G/G$   is compact. So $\+ N_G/GC$ is also compact as its topological quotient. If $M_G$ is interdefinable with a structure in a   finite signature, then $\+ P$ is finite by the second statement of Theorem~\ref{profinite_lemma}.
 \end{proof}

  \begin{remark}  Using the  model theoretic setting of bi-interpretations,   Schlicht and the first author   \cite[Section~3]{Nies.Schlicht:nd}  obtain  results related to \ref{prop: tdlc}.   The group $\Out(G)$ corresponds to the group of self bi-interpre\-tations of $\mathrm{Th}(M)$ up to an appropriate congruence relation, and   $\mrm{range} (R)/\Inn(G)$ corresponds to the group of one-dimensional such bi-interpretations up to this congruence. If $M_G$ is interdefinable with a structure in a finite signature, then this group is clearly finite.    This gives a perhaps more instructive  proof that  $\Out(G)$ is discrete in this case. 
	%	We  hope  that this   work   of the first author with Schlicht will allow us to describe the structure of the groups $N_G/G$ and   $\Out(G)$ in terms of $\mrm{Th}(M_G)$. 
\end{remark}

 \section{A general result involving a Borel property  of open subgroups} \label{ss:PropP}
Recall that by convention,   isomorphisms between topological groups are assumed to be topological.  The following set-up resembles the one in  \cite[Th.\  3.1]{nies}; we refer to that source for background. In particular, the set of closed subgroups of $\Sym(\omega)$ forms a standard Borel space denoted $\+ U(S_\infty)$, and its subset of oligomorphic groups is Borel. 

\begin{definition}\label{def_sub-basic} Let $\+ V$ be a Borel class of oligomorphic groups. Let   $P$  be a  Borel relation picking   open subgroups from groups  $G \in \+ V$. If $P(U,G)$ holds, we say that $U$ is a \emph{$P$-subgroup of} $G$. \begin{enumerate}[(a)] \item   $P$   is called \emph{invariant} if any   isomorphism between groups  $\  G ,  H \in \+ V$ preserves~$P$. 
		\item An invariant Borel  relation  $P $ is  called  \emph{sub-basic  for}  $\+ V$  if   each open subgroup of a group $G \in \+ V$ contains a finite intersection of  $P$-subgroups of $G$.\end{enumerate}    \end{definition}
%This is a Borel property because the operator  $G \to \+ W(G)$ is Borel.  ADD DETAIL
%\noindent Condition  (a) tends to hold for natural properties of open subgroups.  Condition~(b) rules out the property ``being  of the form $G_a$ for some $a \in \omega$" (which satisfies the other two conditions), for  this depends on $G$ as a group of permutations.  Condition (c) says that the $P$-subgroups form a sub-basis of neighbourhoods of the identity.
  
 \begin{assumption} \label{ass:1}  In the following we fix a Borel class $\+ V$ of oligomorphic groups, and assume that $P$  is an  invariant Borel relation  that  is sub-basic for $\+ V$. \end{assumption}
 
For each $G \in \+ V$ we   define a countable structure  $\+ C^P_G$  that is a ``trimmed-down" version  of the meet groupoid $\+ W(G)$ of Definition~\ref{def:meetgr}. 
% It will play the role of $\+ E_M$ in the proof of Theorem~\ref{first_th}.
The domain of  $\+ C^P_G$ consists of  the  cosets of  $P$-subgroups, and in place  of the meet operation it has  merely   non-disjointness relations.   A structure  of form $\+ C^P_G$ is  similar to the structures in \cite[Definition\ 2.15]{reconstruction2} used in the particular  case of groups of the form $\Aut(M)$ where $M$ has weak elimination of imaginaries. However,  to get the duality  between groups in $\+ V$ and their corresponding structures expressed in  Lemma~\ref{the_first_crucial_lemma} below, the partial product operation is needed; this operation was absent from  \cite[Definition\ 2.15]{reconstruction2}.

	\begin{definition}\label{def_expanded_12_extra}     Let $G \in \+ V$.
We form a structure   $\mathcal{C}^P_G$ on the set      \begin{center} $\{ gU: U \text{ is a $P$-subgroup of $G$} \text{ and } g \in G \}$  \end{center}
with the  following.
%	Suppose an invariant  property $\+ P$ of open subgroups is sub-basic for $G$. We define a countable  structure $\mathcal{C}_G$.
	\begin{enumerate}[(a)]
	 	%$\mathcal{CEO}_G := \{ gH: H \in \mathcal{EO} \text{ and } g \in G \}$;
	\item\label{def_expanded_12_unary}\label{dom} \label{cod}    Unary functions $\mrm{Dom}$ and  $\mrm{Cod}$ (which stand for ``domain'' and ``codomain") such that   $\mrm{Dom}(B)  =U$ iff $B$ is a right coset of the subgroup $U$, and $\mrm{Cod}(B)  =V$ iff $B$ is a left coset of the subgroup $V$. 
		\item 
	For each $n$, the  $n$-ary relation   $I_n$    holds of $(B_1, \ldots, B_n)$ if $B_1 \cap \ldots \cap B_n \neq \emptyset$;
	 
	 \item A ternary relation     holds of $(A,B,C)$ if $\mrm{Cod}(A)= \mrm{Dom}(B)$ and $AB= C$. As before in the case of $\+ W(G)$, we  write $A\cdot B=C$ for this. 
\end{enumerate}
\end{definition}
We will omit the superscript $P$ in $\+ C^P_G$  if   $P$ is understood from the context. The next lemma establishes the duality between   groups  $G$ in $\+ V$  and the corresponding structures $\+ C_G$.

	\begin{lemma}\label{the_first_crucial_lemma} Let $G,H \in \+ V$.  If $\aaa$ is   an  isomorphism $G \to  H$, let $R(\alpha) \colon \mathcal{C}_G \to  \mathcal{C}_H$ be the map given by $B \mapsto \alpha(B)$. 
	 	\begin{enumerate}[(1)]
		\item $R(\alpha)$ is an isomorphism of the structures $\mathcal{C}_G $ and $ \mathcal{C}_H$.
	\item Each isomorphism $\theta: \mathcal{C}_G\to \mathcal{C}_H$ is of the form $R(\alpha)$ for a unique $\alpha$.
	%\item Let $\mrm{max}\{k_M, k_N\} \leq k$. Then $\+ C_M(k) \cong \+ C_N(k)$ implies $\mrm{Aut}(M) \cong_{\mrm{top}} \mrm{Aut}(N)$.
	\item\label{crucial_item2} The map   $R$ restricted to   $\mrm{Aut}(G)$ is an isomorphism  
  $\mrm{Aut}(G) \to   \mrm{Aut}(\mathcal{C}_G)$.
  %\gianluca{The notation $\cong_{\mrm{top}}$ is obsolete, no?} YES
	\end{enumerate}
\end{lemma}
	\begin{proof}  We have  $G = \mrm{Aut}(M)$ and $H = \mrm{Aut}(N)$ where $M, N$ are the canonical structures for $G$ and $H$, respectively.

	\n (1) Given that  $\alpha$ is a topological isomorphism, it only remains to check  that the relations  $\mrm{Dom}$ and $\mrm{Cod}$ are preserved in both directions by $R(\alpha)$. For all elements~$B,U,V $ in a structure of the form $\+ C_L $,  we have $\mrm{Dom}(B,U) $ iff $B \cdot B^{-1} = U$, and $\mrm{Cod}(B,V) $ iff $B^{-1} \cdot B = V$. This suffices because by its definition, $R(\aaa)$ pre\-serves inversion. 
	
	\smallskip

	\noindent (2) Let   $ (U_i)_{ i < \omega}$  list the subgroups in $\mathcal C_G$, and let   $V_i:=\theta(U_i)$. Then $ (V_i)_{ i < \omega}$  lists the subgroups in $\mathcal C_H$. We define an isomorphism $\aaa \colon G \to H$ such that $R(\alpha)= \theta$. Given $g \in G$, first we  define the value $\alpha(g)$,  which will be  a function $\omega\to \omega$. Let $A_i = gU_i$ and  $B_i= \theta(A_i)$. For each $n$ we have $\+ C_G \models I_n(A_1, \ldots, A_n)$. So, since $\theta$ is an isomorphism,  there is $r_n \in \bigcap_{i< n} B_i$.
	 Since $\mrm{Cod}(B_i, V_i)$ holds in $\mathcal C_H$, we have $r_n V_i = B_i$ for each $i< n$. 
	Since $P$ is sub-basic for $H$, for each  $b \in N$   there is  a finite set $I \subset \omega$ such that $\bigcap_{i \in I} V_i \subseteq H_{b}$. So \[\alpha(g)(b) \colon = (\lim_n r_n)(b)\] exists for each $b$.  
	
\smallskip \noindent 	We  verify that $\alpha$ is a homomorphism $G \to H$. For the rest of  this proof, to avoid an awkward  buildup of parentheses due to   iterated function applications, we use the following notation: instead of $f(x)$ we write $f.x$ (read as ``$f$ of $x$''). So  $f(x)(z)$ becomes $f.x.z$.

\smallskip \noindent 	First we verify that $\alpha.e_G= e_H$. For $g = e_G$, in the definition of $\alpha.g$  we have $A_i = U_i$ and so $B_i = V_i$;  thus $r_n \in \bigcap_{i< n} V_i$. For any  $x \in \omega$, we have $\alpha.e_G.x = \lim_n r_n.x$. Since  $r_n \in \bigcap_{i< n} V_i$, this  implies $\alpha.e_G.x=x$.

\smallskip \noindent 	Next we verify that for each $g,h \in G$, we  have  $(\alpha.h) \circ (\alpha.g) = \alpha.(hg) $.
	Let   $U'_i = gU_ig^{-1}$ and $V'_i = \theta(U'_i)$. Also let $A_i= \theta(gU_i)$ and $B_i = \theta(hU_i')$. Given $x\in \omega$ write $y= \alpha.g.x$. 
%	Let $n$ be so big that $\bigcap_{i<n} V_i \subseteq H_{(x)}$ and $\bigcap_{i<n} V'_i \subseteq H_{y}$. Fix  $r \in \bigcap_{i< n} A_i$  and  $s \in \bigcap_{i<n} B_i$. We have $r.x=y$, and     $s.y= \alpha.h.y$. 
Let $n$ be so large that $\bigcap_{i<n} V_i \subseteq H_{x}\cap H_{y}$. Fix  $r \in \bigcap_{i< n} A_i$. Then   $r.x=y$. Since $P$ is invariant as defined in \ref{def_sub-basic},  the class of  $P$-subgroups of $G$ is   closed under conjugation. So  let $m \geq n$ be so large that each $U_i$, $i< n$, is of the form $U_j'$ for some $j < m$. Let $s \in \bigcap_{i<m} B_i$.  Then    $s\in \bigcap_{i<n} \theta(hU_i)$, and hence $s.y= \alpha.h.y$.

\smallskip \noindent 	Since $hU_i' \cdot gU_i = hg U_i $ holds in $\+C_G$ and $\theta: \mathcal{C}_G\to \mathcal{C}_H$ is an isomorphism, for each $i$,  $B_i \cdot A_i= \theta(hg U_i)$ holds  in $\+ C_H$. Thus  $s\circ r \in \bigcap_{i<n} \theta(hg U_i)$. By the choice of $n$ and since $\mrm {Cod}(\theta(hgU_i)) = V_i$,  this shows that $(\alpha.h) \circ (\alpha.g)(x)=  s\circ r (x)= \alpha. (hg).x$.   

\smallskip \noindent Letting $h= g^{-1}$, we infer that $\aaa(g) \in \Sym(\omega)$. Then 
	 clearly $\alpha(g)  $ is in $ H$ because  $H$ is a closed subset of $\Sym(\omega)$;   recall here that $\alpha(g)   = \lim_n r_n$ where  $r_n \in H$, for all $n < \omega$.
	
\smallskip \noindent	 Exchanging  the roles of $\+ C_G$ and $\+ C_H$ and using $\theta^{-1} \colon \+ C_H \to \+ C_G$,  we find an inverse of~$\alpha$. Thus $\alpha: G \to H$ is a group isomorphism.  
	 To see that $\alpha$ is continuous at $e_G$ (and hence continuous), it suffices to show that $\alpha^{-1}(H_{x})$ contains an open subgroup, for each $x \in N$. As before let $n$ be such that $\bigcap_{i<n} V_i\subseteq  H_{x}$; then $\bigcap_{i< n} U_i $ is the required subgroup. Similarly, $\alpha^{-1}$ is continuous. 
	 
%\gianluca{Notice the sentence it is now easy to verify that $\alpha$ is an homomorphism!!!}
%\gianluca{The idea is that $g$ is uniquely identified by how it acts by conjugaction on the $U_i$'s together with the cosets and we construct $\alpha(g)$ by hand so it does on the $V_i$'s and its cosets exactly what $g$ does upstairs on the $U_i$'s. As all is controlled by conjugation it seems to me obvious that we have an homomorphism. Agree?}
%\andre{sounds plausible, but there is no actual proof there so far... :)  
%Conjugation of the $U_i$ can only determine $g$ up to the centre which could be nontrivial.  On the bright  side,   the definability part Prop 4.7 can easily include the new relation because this just says that the orbit of the n-tuple of imaginaries for the U and the U' are the same in $M^{\mrm{eq}}$.}
%\gianluca{You convinced me that some form of coset composition is needed. But I am not sure about the details. Can you show me what exactly we need to add in 4.5 and how it is used here? You seem to see this more clearly than me and so this seems like a reasonable assumption. I personally would still leave the predicates $\sim_n$ as it makes clearer the fact that $\alpha(g)V_i (\alpha(g))^{-1} = V'_i$ which I think that we want.}

%	By construction we have that $\alpha(g)V_i (\alpha(g))^{-1} = V'_i$, for all $i < \omega$. In fact if $g, h \in G$, then we have the following:
%	$$\alpha(gh^{-1})V_i (\alpha(gh^{-1}))^{-1} = \alpha(g)\alpha(h)^{-1}V_i \alpha(h)\alpha(g^{-1}).$$
%	 It is also easy 

\smallskip

\noindent  (3)  For each $\alpha ,\beta \in \Aut(G)$, we have   $R(\alpha \beta^{-1})= R(\alpha) R(\beta)^{-1}$ by the definition of~$R$, so $ R$ is an isomorphism of groups.  To see that $R$ is continuous (and hence a homeomorphism using   \cite[2.3.4]{gao}), note that a basic neighbourhood of the identity in $\Aut(\+ C_G)$ has the form $\{\theta \colon \bigwedge_{i=1}^n \theta(A_i)= A_i\}$, where $A_1, \ldots, A_n \in \+ C_G$. 
Its preimage under $R$  is open in $\Aut(G)$ according to Definition~\ref{df:nbhd}.  \end{proof}

%%%%%%
\n By Lemma~\ref{the_first_crucial_lemma}(3) and the definition of its topology,  $\Aut(G)$ acts continuously on $\+ C_G$ by automorphisms, namely, $\Phi \cdot A =  \Phi(A)$ where $\Phi \in \Aut(G)$ and $A \in \+ C_G$. Under an additional hypothesis,      the continuous subaction of the closed subgroup $\Inn(G)$ on $\+ C_G$  is    oligomorphic.
\begin{lemma}  \label{lem:frak} Let $G\in \+ V$. Assume that  there are    only finitely many conjugacy classes of $P$-subgroups in $G$. Then the  canonical action by automorphisms  of $\Inn(G)$ on $\+ C_G$    is oligomorphic. \end{lemma}
 
\begin{proof}  
By \ref{rem:top Inn}, $\Inn(G) $ is topologically isomorphic to $ G/Z(G)$, and hence       Roelcke precompact. Thus, by   \cite[Theorem 2.4]{Tsankov:12} it suffices to show that the action of  $\Inn(G)$ on $\+ C_G$   has only  finitely many $1$-orbits. 

%Consider an  element $A$ of $\+ C_G$, we want to describe its orbit up to finitely many by 

\smallskip \noindent Recall that for a pair $U,V$ of conjugate open subgroups, by~\ref{fact:conjugates} there are $[N_U:U]$  right cosets  of $U$ that are left cosets of $V$,  and $[N_U:U]$       is finite as $G$ is Roelcke precompact. 
So it suffices to show that there are only finitely many $\Inn(G)$-orbits of   pairs $(U,V)$ where $U,V$ are conjugate $P$-subgroups; the $\Inn(G)$-orbit of $A \in \+ C_G$ is then given by the orbit of $U,V$ where $U$ is the domain  and $V$ the codomain of  $A$, together with the information which right coset of $U$ the element $A$ belongs to.

\smallskip \noindent    By hypothesis $U,V$  are of the form $G_\aaa$ and $G_\beta$, where  $\aaa, \beta$ are both elements of one of   finitely many 1-transitive sorts  of $M^\mrm{eq}$. The structure $M^\mrm{eq}$ has only finitely many  orbits of such pairs $\aaa, \beta$.
 % by \ref{count cosets}  the orbit of such a  coset is now given by  $k$ many possibilities, where $k$ only depends on $S$. 
 \end{proof}
%Every $g \in G$ induces an automorphism $\widehat{g} \in \mrm{Aut}(\mathcal{C}_G)$ %(cf. \ref{C+_G_def})
%in such a way that the map $g \mapsto \widehat{g}$ is an homomorphism.  The map $\widehat{g}$ is defined by letting, for $H$ an essential subgroup of $G$ and $k \in G$, $\widehat{g}(kH) = gkg^{-1} gHg^{-1}$. The kernel of this map coincides with the centre $Z(G)$,   so we have an action of $\Inn(G) \cong G/Z(G)$  on $\mathcal{C}_G$. This action  is faithful and  strongly continuous, in the sense of \cite[Sec.~3.3]{coarse}. 

 Now let $\mathcal{C}^+_G$ be  the expansion of $\mathcal{C}_G$ obtained by naming the $n$-orbits of the action of $\Inn(G)$ on $\mathcal{C}_G$ by predicates, for all $n \geq 1$.  Note that $\+ C_G$ is a reduct of $\+ C^+_G$.
  \begin{fact} \label{useful_prop} Suppose that $G$ is as in Lemma~\ref{lem:frak}. Then  $\+ C^+_G$ is $\aleph_0$-categorical,
    and so $\+ C_G$ is $\aleph_0$-categorical as well. 
\end{fact}

	\begin{proof} By~\ref{lem:frak} the action by automorphisms of $\Inn(G)$ on  $\+ C_G$ is oligomorphic. By definition, each such automorphism is also  an automorphism of $\+ C_G^+$. Thus $\Aut(\+ C_G^+)$ is oligomorphic, and hence $\+ C_G^+$ is $\aleph_0$-categorical.
	\end{proof}
Groups $G \leq_c \Sym(\omega)$ such that $G$ is (topologically) isomorphic to an oligomorphic group are called \textit{quasi-oligomorphic} in  \cite[Section 4.3]{coarse}. The authors  show that this property is Borel, and the complexity of their  isomorphism relation is Borel equivalent to the isomorphism relation between oligomorphic groups. 
 
\begin{theorem}\label{second_th-rec}  In addition to  Assumption~\ref{ass:1}, suppose  that for each $G \in \+ V$, there are only finitely many conjugacy classes of $P$-subgroups. 

\begin{enumerate}[(1)] \item 	Topological isomorphism on  $\+ V$ is smooth. 

\item   $\mrm{Aut} (G)$ is quasi-oligomorphic  for each  $G \in \+ V$.

%\item  
%Where $G = \mrm{Aut}(M)$, we have $\mrm{Aut} (G)/\Inn(G)\cong_{\mrm{top}} \mrm{Aut}(\+ C_G)/\mrm{Aut}(\+ C^+_G)$.
\item $\mrm{Out} (G)$ is profinite  for each  $G \in \+ V$.
  \end{enumerate}
\end{theorem}

\begin{proof} (1)   The structures with domain $\omega$ in a fixed countable  signature $L$  can be seen as a standard Borel space $\mrm{Mod}(L)$ \cite[Section 3.5]{gao}.  By  \cite[Section 3.3]{nies}, there is a Borel function taking an oligomorphic  group $G$ to a bijection $ \nu_G\colon \+   \omega \to \mrm{dom} (\+ M(G))$  (this uses a  result of Lusin-Novikov    in the version of  \cite[18.10]{Kechris:95}).  Via this bijection we can view $\+ M(G)$ as a structure with domain $\omega$ in such a way that  the map $G \mapsto \+M(G)$ is Borel. The structures $\+ M(G)$ and the meet groupoid $\+ W(G)$ are interdefinable in a uniform way \cite[Remark 2.16]{coarse}, so the map $G \to \+ W(G)$ is also Borel in this sense. Then, since $P$ is a Borel relation, the map $G \mapsto \+ W^P(G)$ is Borel,  where we define $\+ W^P(G)$ as the structure $\+ W(G)$ extended by a unary predicate  saying that a subgroup satisfies $P$.  
 The  domain of the structure $\+ C^P_G$  can be first-order defined in   $\+ W^P(G)$  as the set of cosets of subgroups that satisfy the property $P$. 
 
 We claim  that    each relation    and operation of  $\+ C^P_G$ in Definition~\ref{def_expanded_12_extra} can also be   defined in   $\+ W^P(G)$ by a first-order formula that does not  depend on $G$ (and also not on  $P$).    For (a)  in \ref{def_expanded_12_extra},  this holds  because $\mrm{Dom}(B)=U$  iff $B \cdot B^{-1} = U$ and  $\mrm{Cod}(B)=V$ if $B^{-1}\cdot B = V$. For (b)  in \ref{def_expanded_12_extra}, one uses that $\+ W^P(G)$ contains the binary intersection operation and a constant for $\emptyset$, and for (c) it is immediate.
 
  Viewing $\+ W^{P}(G)$  as a structure with  domain $\omega$ as above, 
 our  claim   shows that the map $G \mapsto \+ C^P_G$ is   Borel.  By Lemma~\ref{the_first_crucial_lemma} and Fact~\ref{useful_prop}, the Borel map $G \mapsto \mrm{Th}(\+ C^P_G)$ now establishes smoothness of $\cong$ on $\mathcal{V}$. 
%  \gianluca{I can't find a proof that the class $\mathcal{B}_P$ is Borel. Please provide one. OK see above. 
%  \\That's not how I would have written things. I would have assume that $\mathcal{B}_P$ is Borel. As we don't know that $\mathcal{B}_P$ is the collection of $\mathfrak{G}$-finite groups where $P$ is being essential. So we can't directly use this to infer the second main result.}  \andre{it's a little more work to do this. I'll do it once you are fine with the P thing. We shouldn't assume that $B_P$ is Borel if this really from that  $P$ is Borel.} \gianluca{I am fine with the $P$ thing, so please go ahead. I now see the logic of your way of phrasing things. In particular I agree that this gives a slightly more general result. $\mathfrak{G}$-finiteness should be mentioned as the canonical example of when essential subgroups exist, I think that you agree on this, this is why you left 4.8 below.}

\smallskip 
\n  (2)  By    Lemma \ref{the_first_crucial_lemma}(3),   $\mrm{Aut} (G) \cong \mrm{Aut}(\+ C^P_G)$.  Now apply Fact~\ref{useful_prop}.

\smallskip

\n 
%For (3) let $F \colon \Inn(G) \to \Aut(\+ C_G)$ be the embedding given by the subaction of $\Inn(G)$ on $\+ C_G$  by automorphisms. Thus $F$ is the restriction of $R$ to $\Inn(G)$.  
 (3) Let $R$ be as in Lemma~\ref{the_first_crucial_lemma} for $H=G$. By the definition of $\+ C_G^+$, the closure of $R(\Inn(G))$ equals  $\Aut(\+ C_G^+)$.
%   \gianluca{Please provide more details for the previous sentence. Please don't say "I don't know what has to be added". If I ask it means that either it is not completely clear or a sentence can help.}\andre{This is a general fact which works for actions where the image in the premutation group is not closed. We've colored all the Inn(G) orbits to get $\+ C_G^+$, so the closure  is Aut of it.  Not sure about a ref but this is so basic I think...} \gianluca{OK, fine. I of course believe it and I don't know of a reference either. It's just that if the referee asks we have to find a way to convince him/her of this.} 
 By Theorem~\ref{prop: Inn closed},  $\Inn(G)$ is closed in $\Aut(G)$; since  $R$ is a topological isomorphism, $R(\Inn(G))$ must in fact be closed,  and hence equal to $\Aut(\+ C_G^+)$. Also  $\Aut(\+ C_G^+)$ is normal in $ \Aut(\+ C_G)$ because $\Inn(G)$ is normal in $\Aut(G)$.  Therefore
$$\mrm{Aut} (G)/\Inn(G) \cong \mrm{Aut}(\mathcal{C}_G)/\mrm{Aut}(\mathcal{C}^+_G).$$ 
The statement  (3) now follows by  invoking Lemma~\ref{profinite_lemma+}. 
% and \ref{profinite_lemma}, as $\+ C_G$ is $\aleph_0$-categorical. %Concerning (4), this is what (3)(ii) says.
\end{proof}

	  We use  the foregoing result to show that in fact,  (3)   implies (2) for \textit{any}  oligomorphic group  $G\leq_c \Sym(\omega)$.  
	\begin{proposition}
		Suppose  that $G$ is oligomorphic and $\Out(G)$ is profinite. Then $\Aut(G)$ is quasi-oligomorphic.
	\end{proposition}
	
	\begin{proof}
		We will apply Theorem~\ref{second_th-rec}  to the Borel class $\+ V= \{G\}$ and the property $P(U,G)$ saying that $U$ is automorphic to $G_a$ for some $a \in \omega$.   
		
\smallskip \noindent		The group 	$\Aut(G)$ acts continuously on the (discrete) space of open subgroups of~$G$:  for this it suffices to verify that $\{ \aaa \in \Aut(G) \colon \aaa(U)= U\}$ is open for each  $U \leq_o G$, which follows from the definition of the topology on $\Aut(G)$ in~\ref{df:nbhd}. Therefore $\Out(G)$ acts continuously on the discrete space of conjugacy classes of open subgroups of~$G$.  Since $\Out(G)$ is profinite, every orbit of this action is finite.  There are only finitely many conjugacy classes of subgroups of the form $G_a$. The union of the  orbits of these conjugacy classes  under the $\Out(G)$ action consists  therefore of  finitely many conjugacy classes $\+ C_1, \ldots, \+ C_m$.  The property $P(U,G)$ holds precisely  when $U \in \bigcup_r \+ C_r$. Thus  $P$ is Borel and   there are only finitely many conjugacy classes of $P$-subgroups. 		
	The Assumptions~\ref{ass:1}  are satisfied, so $\Aut(G)$ is quasi-oligomorphic  by  (2) of Theorem~\ref{second_th-rec}. 	
\end{proof}

\section{Two smooth classes} \label{smooth section} We apply the general criterion of Theorem~\ref{second_th-rec} to two Borel classes of oligomorphic groups, thereby showing that their isomorphism relation on the class is smooth, the automorphism group of each such group is profinite, and hence  its  automorphism group is quasi-oligomorphic (i.e., isomorphic to an oligomorphic group).

An action of a  group $G$   on a set $X$   corresponds to a  homomorphism $G \to \mathrm{Sym}(X)$.   One says that $(G,X)$  is \emph{faithful }if this homomorphism is 1-1, and  $(G,X)$ is \emph{closed} if its range is closed in $\Sym(X)$. 
\subsection{Groups with no algebraicity}  \label{s: no alg}
   Let $G\le \Sym(\omega)$ be oligomorphic. Recall from~\ref{df: MG} that   $M_G$ denotes the canonical structure $M$ such that $\Aut(M)= G$. A countable  structure $M$ has \emph{no algebraicity} if $\acl(A) = A$ for each $A \sub M$.  
	\begin{definition} \label{no alg}   $G$    has \emph{no algebraicity} if   for each finite set $S \subseteq \omega$, the action of $G_{(S)}$ on $\omega$  has no finite $1$-orbits on $\omega \setminus S$.  Equivalently, the structure $M_G$ has no algebraicity. \end{definition} 
Note that this property depends on $G$ as a permutation group.
  It   is clearly a Borel property of structures with domain $\omega$ to have no algebraicity. Since  the operation $G \mapsto M_G$ is Borel~\cite[Secton 1.3]{coarse},  the class of groups with no algebraicity is Borel.

  In the following, let $N$ be the substructure of $M^{\mrm{eq}}$ with domain  a sort $S= D/E$.   We say that $N$ is  \emph{dense} in  $M^{\mrm{eq}}$ if   for every $b \in \mrm{M}^{\mrm{eq}} $ there is a finite subset $A \subseteq N$ such that $b$ is definable from $A$ in $M^{\mrm{eq}}$. (This property  is called ``regular" in \cite{rubin}.)
% 
%	\begin{notation}[{\cite[Definition~2.5]{rubin}}] For   a structure $M$ and  a complete $n$-type $p$ in $M$, by  $M_p = p(M)$ one denotes $\{\bar{a} \in M^n : \mrm{tp}_M(\bar{a}) = p\}$.
%\end{notation}
Note that  $\Aut(M)$ acts naturally on $N$.

\begin{fact}[{by \cite[Lemma~2.10]{rubin}}]\label{Rubin_regular} Suppose that $(\mrm{Aut}(M), N)$    is faithful,  and closed. Then $N$ is a dense
substructure of $M^{\mrm{eq}}$.
\end{fact}

	\begin{fact}[{by \cite[Lemma~2.11]{rubin}}]\label{Rubin_hammer} Suppose that    $E$ is non-degenerate   in the sense of Definition~\ref{df:degen}.  Suppose that $M$ has no algebraicity, $N  $  has no algebraicity, and $N$ is     dense in $M^{\mrm{eq}}$. 
	 Then %$M = N$.
	 $D=M$ and $E$ is the identity equivalence relation.
\end{fact}
 
 The next result  together with  Theorem~\ref{second_th-rec} yields Theorem~\ref{great theorem} for the first class.
\begin{lemma}  \label{lem: no alg}The hypotheses  of Theorem~\ref{second_th-rec} are satisfied in the  following setting:   

\begin{itemize} \item   $\+ V$ is the class of groups with no algebraicity. \item  $P(U,G)$ holds if $G \in \+ V$ and $U = G_{  a }$ for some   $  a \in \omega$.  \end{itemize}
\end{lemma}
\begin{proof}  To verify that $P$ is   Borel, we use the standard methods surveyed in Kechris et al.~\cite[Section~2]{nies}. By $ \+ U(S_\infty)$ they denote the Borel  space of closed subgroups of $S_\infty$, which is a subspace of the usual Effros space $\+ F(S_\infty)$.  They verify that inclusion is Borel on this space. For $H \in \+ U(S_\infty)$, by  $T_H$ they denote  a tree of strings over $\omega$ describing $H$;  the map $H \mapsto T_H$ is Borel.   For the Borelness of $P$ it  suffices to  show that the map $\omega \times  \+ U(S_\infty) \to \+ U(S_\infty)$ that sends $(a,G)$ to $G_a$ has  Borel graph.  For $H \in  \+ U(S_\infty)$, we have \begin{center}
$H= G_a \Leftrightarrow H \sub G \land \forall \beta \in T_H [|\beta|  > a \rightarrow \beta(a) = a]  $.
	\end{center}  This condition is clearly Borel. 

For each $G \in \+ V$,  the number of conjugacy classes of $P$-subgroups equals the number of 1-orbits of $G$, which is finite. Clearly, the $P$-subgroups generate the filter of open subgroups of $G$.  It remains to show that $P$ is invariant for groups in $\+ V$ in the sense of Definition~\ref{def_sub-basic}(1).  

\smallskip 

\begin{claim}   {\it Suppose $G,H \in \+ V$.  For each  isomorphism $\Phi \colon G\to H$  there is  a bijection $f\colon M_G \to M_H$ such that $\Phi(g)= f \circ g \circ f^{-1}$ for each $g \in G$. } \end{claim}  

\smallskip

\n Assuming the claim,   $\Phi(G_a) = H_{f(a)}$ for each $ a\in \omega$. Thus $P$ is preserved.  

\n \emph{We  verify the claim using~\ref{Rubin_hammer}}. Note that $\Phi$   corresponds to a bi-interpretation of $M_G$ and $M_H$ by \cite[Cor.\ 1.4]{biinte}.  In particular, there are a sort $S= D/E$ of $M_G$ and an onto map $f \colon D \to M_H$ with kernel $E$ such that $\Phi = \Aut(f)$ in the sense of \cite[1.3]{biinte}. Note that $G$ acts on $M_H$ via the given interpretation of $M_H$ in $M_G$.  

Let $N$ be the substructure of $M_G^{\mrm{eq}}$ with domain $S$.  Since $M_G$ and $M_H$ are bi-interpretable,   the   action of $\mrm{Aut}(M_G)$ on $M_H$ is   faithful and closed. (This is easy to verify  directly, but  also follows from the independently established Proposition~\ref{modelth background} below.) Hence   the action of $G$ on $N$ is   faithful and closed as well.  

We can assume   $E$ is non-degenerate (Definition\ \ref{df:degen}): otherwise, say $E$ contains the equivalence   relation $E_{\{0, \ldots, k-2\}, D}$   expressing  that   two $k$-tuples over  $D$ agree on the first $k-1$ elements. Let $D'=p(D)$ where $p$ is   the definable projection $M^k\to M^{k-1} $ that erases the last component. Then $D'$ is definable, and $p$ yields an $M$-definable bijection $D/E \to D'/p(E)$. So we can replace $N$ with the structure $N'$ that has domain $D'/p(E)$. 

 Clearly $(G,M_H)$ and $(G,N)$ are isomorphic as permutation groups. So one has   algebraicity iff the other does. By hypothesis, this implies that $N$   has no algebraicity. 
%\andre{is this right?  It is right if $M_H$ and $N$ are interdefinable via the bijection induced by $f$. Otherwise, there could be more relations on $N$ that $M_H$ doesn't see, so $N$ could have algebraicity. I think I can see a proof that they are inter definable, but it would be tedious to write here.}

By~\ref{Rubin_hammer}, $D=M$ and $E$ is the identity equivalence relation. Thus $f \colon M_G \to M_H $ is a bijection. Since     
$\Phi = \Aut(f) $  in the sense of \cite[p.\ 66]{biinte}, this shows $\Phi(g)= f \circ g \circ f^{-1}$,   as required for the claim.
\end{proof}

We also offer  a direct approach to     the case of oligomorphic groups~$G$ with no algebraicity. (This is absent from the published paper in the Beijing J. of Pure and Applied Mathematics, 2026.) It  gives a more concrete description of $\Aut(G)$ and $\Out(G)$.  It can potentially be used to describe these groups for particular $G$, similar to the last section of \cite{reconstruction2}. 	The following will be needed.

\begin{fact}\label{no_alg} If $M$ has no algebraicity, then   the centraliser of  $G=\mrm{Aut}(M)$ in $\Sym(\omega)$ is trivial. In particular,    $\mrm{Aut}(M)$ has trivial center. 
\end{fact}
\begin{proof} Any permutation $h$ in the centraliser of $G$  is invariant under the action of $G$ on $M$, and hence definable as a relation. So $h$ is the identity. \end{proof}

\begin{theorem} \label{th:no alg detail} Let $G$ be an oligomorphic group without algebraicity, and let $M= M_G$. We have   
	(1)   $\mrm{Aut}(G) \cong \mrm{Aut}(\mathcal{E}_M)$ and 
	(2) 	$\Out(G) \cong \mrm{Aut}(\mathcal{E}_M)/G$. \end{theorem}
Note that the topological group  $\mrm{Aut}(\mathcal{E}_M)/\mrm{Aut}(M)$ is profinite by \ref{profinite_lemma}. So we have another proof that $\Out(G)$ is profinite.
\begin{proof} 	
	\smallskip \noindent (1)     Fix  $a_1, ..., a_k \in M$ as  representatives of the  $1$-orbits of $G$ acting on $M= M_G$. 
	Given  $\beta \in  \mrm{Aut}(G)$, we  will  define a corresponding map $L_\beta \colon M \to N$ where $N \subset M^{\mrm{eq}}$ is a submodel given by finitely many sorts. Our goal is    to show    that $N=M$ and hence  $L_\beta \colon M\to M$.% \in \mrm{Aut}(\mathcal{E}_M)$.

	\smallskip 
	
	\noindent The construction of $L_\beta$ is similar to the one in the proof of~\cite[Thm.\ 1.2]{biinte}. 
	\noindent For $i \leq k$ let $b_i \in M^{\mrm{eq}}$ be such that $G_{b_i}= \beta(G_{a_i})$. By~\ref{evans_fact} we may assume that $b_i $ is in a sort $D_i/E_i$ where $E_i$ is nondegenerate. Define \[L_\beta(g\cdot a_i ) = \beta(g) \cdot b_i.\] (This is well-defined because $g\cdot a_i = h\cdot a_i $ iff $g^{-1}h \in G_{a_i}$ iff 
	$ \beta(g) \cdot b_i = \beta(h)(b_i)$.) 
	
	\smallskip
	\noindent  The permutation groups $(G, M)$ and $(G, N)$ are isomorphic via  $\langle g,a \rangle \mapsto \langle \beta(g), L_\beta{(a)} \rangle$. Hence, $(G, N)$ has
	only finitely many orbits,    and  is faithful and closed. Thus, by \ref{Rubin_regular}   $N$ is a dense substructure of $M^{\mrm{eq}}$. 
	Further,  $N$ has no algebraicity since this holds for $M$, and is a property that can be seen from $(G,M)$, as indicated  in Definition~\ref{no alg}.
	Using that the sorts of $N$ are non-degenerate, by \ref{Rubin_hammer}, $N= M$, so that $L_\beta \colon M \to M$.
	
	\smallskip
	\noindent We now verify that we have  a (topological)   isomorphism $ L \colon \Aut(G) \to  \mrm{Aut}(\mathcal{E}_M)$ given by  $  \beta \mapsto  L_\beta$. We   mainly employ the two conditions (a) and (b) below: 
	
	\begin{enumerate}[(a)]
		\item   $G_{L_\beta(A)} = \beta(G_{a})$ for each $a \in M$.
		\item  For each $k \geq 1 $ and $\bar a, \bar b \in M^k$, we have  $G_{\bar a}= G_{\bar b}$ iff $\bar a=\bar b$.  In particular, $h \cdot \bar a = \bar c$ iff $ hG_{\bar a} h^{-1}  = G_{\bar c}$ for each $h \in \Aut(\+ E_M)$ (the normaliser of $G$ in $\Sym(\omega)$). 
	\end{enumerate}
	Condition (a) is easy to check from the definitions. Condition (b)  holds because  of  \ref{fac2} and  since  $M$ has no algebraicity. 
	\n From (a) we have that  $L(\beta^{-1})= (L_\beta)^{-1}$. In particular, $L_\beta $ is a permutation of $M$, and $\beta \to L_\beta$ preserves inverses.

	\smallskip 
	
	\noindent We  verify  that  \emph{$L_\beta$ is  an automorphism of $\mathcal{E}_M$}. Fix $n$ with $0 < n < \omega$.  Using the notation of  the definition of $\mathcal{E}_M$ in \ref{def_EM}, we need to show that  \begin{center} $\bar{a} \sim_n \bar{b}$ iff $L_\beta(\bar{a}) \sim_n L_\beta(\bar{b})$,  for each $\bar{a}, \bar{b} \in M^n$.\end{center} If $\bar{a} \sim_n \bar{b}$,   there is $g \in G$ such that $g(a_i) = b_i$ for all $i \in [1, n]$.
	%, where we let $\bar{a} = (a_1, ..., a_n)$ and $\bar{b} = (b_1, ..., b_n)$. 
	Clearly    $gG_{a_i}g^{-1} = G_{b_i}$. As $\beta \in  \Aut(G)$, this implies  $    \beta(g) \beta(G_{a_i}) \beta(g)^{-1} = \beta(G_{b_i})$ and so by (a) and (b),  $\beta(g)\cdot L_{\beta}(a_i) = L_{\beta}(b_i)$. %, by how $L_{\beta}$ was defined. 
	Thus $\beta(g) \in G$ is a witness of the fact that $L_\beta(\bar{a}) \sim_n L_\beta(\bar{b})$, as desired. The converse  implication  is proved analogously,  using $\beta^{-1}$ instead of $\beta$.
	
	\medskip
	
	\noindent To see that \emph{$L$ is a  group homomorphism}, we note that for $\gamma, \beta \in \Aut(G)$, by (a) we have $G_{L_{\gamma \circ \beta}(A) }= G_{L_\gamma(L_\beta(A))}$; this suffices by (b) and since $L$ preserves inverses. 
	%  \andre{Showing that $L(\beta)$ is an automorphism of $\+ E_M$ isn't as straightforward as you claim.  It needs \ref{fac2} but even then, $G_A$ doesn't remember the order of the elements}
	%  \gianluca{I both disagree and do not understand your comment, see the proof I wrote above. What is wrong with it?}
	%OK sorry, I got confused late at night.
	
	\medskip
	
	\noindent To see that \emph{$L$   is  onto},  given $r\in \mrm{Aut}(\mathcal{E}_M)$, define $\beta \in\Aut(G) $ by $\beta(g) =rgr^{-1}$; then $G_{L_\beta(A)}= \beta(G_a)=rG_ar^{-1}=  G_{r(A)}$ for each $a \in M$, so that $L_\beta = r$.   
	
	\medskip
	
	\noindent  To show that  \emph{$L$ is 1-1}, let   \[K = \mrm{ker}(L)=  \{\beta \in \Aut (G) : \beta(G_{a}) = G_{a}, \text{ for all } a \in M\}.\]  Note  that  $K \cap \Inn(G) = \{ e\}$, because,  for every $g\in G$, if conjugation by $g$ is  in $ K$ then   for every $a \in M$ we have  $G_{g(A)} = gG_{a}g^{-1} = G_a$, so that $g$ is the identity by~(a). As $\Inn(G)$ is always normal in $\Aut(G)$, we conclude  that $[K,\Inn(G)]=\{e\}$. An elementary group theoretic fact recalled  as \ref{groups_lemma} below,  together with \ref{no_alg},   now shows that the kernel $K$ of $L$ is trivial. 
	
	\medskip
	
	\noindent It remains to show that \emph{$L $ is continuous} with respect to the Polish topology on $\Aut(G)$ discussed in Section~\ref{ss: Polish}; in this case, $L$ is a homeomorphism   (see, e.g., \cite[2.3.4]{gao}). A   system of neighbourhoods of the identity in $\Aut(\+ E_M)$ is given by the subgroups $\{ r \colon \bigwedge_{i=1}^n \, r (a_i) = a_i\}$, where $a_1, \ldots, a_n \in M$. By (a) and (b), the  pre-image under the map  $L$ of such a group  is  $\{\beta \colon \, \bigwedge_{i=1}^n \beta(G_{a_i}) = G_{a_i}\}$, which is open in $\Aut(G)$. 
	%\gianluca{Yes, so things work :)! Here you have in mind that $M = \{a_i : i < \omega\}$? Maybe it is clear from the notation above but you may recall it. This is just an expositional remark. And minor, of course.}
	%\andre{no, I meant to denote any n-tuple of elements}

	\smallskip
	\noindent (2) The isomorphism  $L: \Aut(G) \to  \mrm{Aut}(\mathcal{E}_M)$  induces an isomorphism between the closed normal subgroups $\Inn(G)$ and $G= \Aut(M)$:  For  $r \in  \Aut(\+ E_M)$,  the pre-image $L^{-1}(r)$ is the map $g \mapsto g^r$. Then $r \in \Aut(M)$ iff this map is an inner automorphism.  
	% So it remains to show that $F\colon \mrm{Aut}_{\mrm{top}}(G) \cong_{\mrm{top}} \mrm{Aut}(\mathcal{E}_M)$ is continuous. 
	%\todov{Debt from Andre: argue that $\mrm{Aut}_{\mrm{top}}(\mrm{Aut}(M)) \cong_{\mrm{top}} \mrm{Aut}(\mathcal{E}_M)$.}
\end{proof}

\begin{lemma}\label{groups_lemma}  Let $G$ be a   group,  and suppose that an automorphism  $\beta$ of $G$ centralises  $\Inn(G)$. Then  
	%becomes trivial in $\Aut(G/Z(G))$; in other words, 
	for   each $g \in G$ we have $\beta(g)g^{-1}  \in Z(G)$. \end{lemma}
\begin{proof}   Let $\hat g \in \Aut(G) $ denote conjugation by $g$. For each $h \in G$ we have  
	\[ (\beta(h))^{\beta(g)}= (\beta \circ  \hat g)(h)=   ( \hat g \circ  \beta)(h)= (\beta(h))^g. \]
	Thus $\beta(g) g^{-1}\in Z(G)$ as required. (The converse implication also holds.)
\end{proof}
 The following   fact is a well-known observation of Hrushovski; see~\cite[Ch.\ 1, Lemma~1.1.1]{Evans.etal:97} where it is proved in the language of permutation groups. In the subsequent  remark we use the  proof to show that the hypothesis that both $M$ and $N$ have no  algebraicity is necessary in~\ref{Rubin_hammer}. %(This is   not be needed for the proof of \ref{first_th}. )
 \begin{fact}   \label{1-transitive_fact} Let $M$ be a    countable $\aleph_0$-categorical structure. There is a countable $\aleph_0$-categorical structure $N$ such that $M$ and $N$ are bi-interpretable,    and     $N$ has only one $1$-orbit.
 \end{fact}

 \begin{proof} Let   $0 < n < \omega$ be the number of $1$-orbits of $M$. We may replace $M$ by the canonical structure for $\Aut(M)$, by naming its  $n$-orbits with predicates of arity $n$. In particular, let $P_1, ..., P_n$ be the predicates naming the $1$-orbits. Let  
 	\begin{eqnarray*} D &=&  \bigcup_{\sigma \in S_n} \prod_{i=1}^n P_{\sigma(i)} \\
 		E &= &   \{ \langle \bar x , \bar y  \rangle\in D^2   \colon \exists \sigma \in S_n \forall i \, [x_i  = y_{\sigma(i)}]\}. \end{eqnarray*}
 	%	\begin{enumerate}[(1)]
 		%	\item $\bar{a} = (a_1, ..., a_n)$ is a tuple with no repetitions;
 		%	\item $\bar{a}$ is made of elements belonging to different $P_i's$, that is, for $i, j, k \in [1, n]$, if $a_i \in P_j$ and $k \neq i$, then $a_k \notin P_j$.
 		%	\end{enumerate}
 	The structure $N$ induced by $M^{\mrm{eq}}$ on   $D/E$ is  what we are looking for.  Notice that   $N$ is a dense subset of $M^{\mrm{eq}}$.
 	%\gianluca{Are you implicitly using that dense + $1$-transitive implies bi-interpretable? As written it is not clear what is going on. Notice that this is proved also in Evan's book. I had given a reference to this, see \cite[Lemma~1.1.1]{Evans.etal:97}. What they do there is that they take a tuple $\bar{a}$ as above and then they take the sort which is made of the orbit of $\bar{a}$ and observe that the action is closed and faithful. From a quick look it seems that what we are doing is the same but as written this proof is not clear = double definition of $D$ and not clear why from dense it follows bi-intepretability. As a matter of exposition, I mean.}
 	%
 	%\andre{Their Lemma 1.1.1 proof  is permutation group theoretic, so not so close to what we are doing. I only slightly improved the notation of what you had (see old version below). You can add more detail if you like. In fact what you have is not the same as in their 1.1.1. , their D is without the union over $S_n$. Also, the $n$ should be $k$ if we keep it.}
 \end{proof}

 \begin{remark}\label{the_transitivity_remark} In the context of the proof of \ref{1-transitive_fact}, if $M$ has more than one $1$-orbit, then $N$ has algebraicity, i.e., there is finite $A \subseteq N$ and \mbox{$a \in N \setminus A$ s.t. $a \in \mrm{acl}_N(A)$.}
 \end{remark}
 \begin{proof} To save on notation, we   show this for the case that  $M$ has only  two distinct infinite $1$-orbits, denoted  $P_1$ and $P_2$. Given  $\langle a,b \rangle \in D$, we denote its equivalence class by $\{a,b\}$.  On $N$ we consider the undirected graph with the    edge relation \begin{center} $\{a, b\} R \{c, d\}$ iff $|\{a, b \} \cap \{c, d \}| = 1$.  \end{center} This relation is definable because  $N  $ inherits all the definable relations from $M^{\mrm{eq}}$. 
 	%	Notice   that four distinct vertices of $(N, R)$ induce a complete subgraph iff they pairwise share the same    element from $M$.
 	Fix  $x \in P_1 $ and $y \in P_2$. Furthermore, fix pairwise distinct $x_i \in P_1\setminus \{x\}$ and $y_i \in P_2 \setminus \{y\}$, for $i =1,2$. Define  $a_i = \{x, y_i\}$ and $b_i = \{x_i, y\}$.    Now  $c = \{x, y\} $ is the unique  vertex connected  to each vertex in $A=\{a_1, a_2, b_1, b_2\}$. Thus $c  \in \mrm{dcl}_N (A) \setminus A$. 
 	%	 , $y_1, y_2, y_3, y_4 \in P_2 $ be so that $y_i \neq y_j$ if $i \neq j \in [1, 4]$ and define, for $i \in [1, 4]$, $a_i = \{x_*, y_i\}$. Similarly, let $y_* \in P_2 \setminus \{y_1, ..., y_4 \}$, $x_1, x_2, x_3, x_4 \in P_1 \setminus \{x_*\}$ be so that $x_i \neq x_j$ if $i \neq j \in [1, 4]$ and define, for $i \in [1, 4]$, $b_i = \{y_*, x_i\}$. Let $A_0 = \{a_1, ..., a_4, b_1, ..., b_4\}$. Then $c = \{x_*, y_*\} \not\in A_0$ but $c \in \mrm{acl}_N(A_0)$. This is maybe best seen in $M^{\mrm{eq}}$. To this extent, in $M^{\mrm{eq}}$ consider the following equivalence relation on $D(M^4)$ (recall the notation from the proof of \ref{1-trnasitive_fact}): $\bar{a} E' \bar{b}$ iff $\bar{a} E \bar{b}$ and $\bar{a}\bar{b}$ form a complete graph in $(N, R)$. Notice now, that $N/E'$ canonically corresponds to $M$. Thus, from the perspective on $M^{\mrm{eq}}$ fixing $A_0$ pointwise means fixing both $x_*$ and $y_*$, which in turn means fixing $\{x_*, y_*\}$, i.e., $c \in \mrm{dcl}_N(A_0)$, as desired.
 \end{proof}

\subsection{Finitely many essential subgroups up to conjugacy}

 Membership  of $G$ in our  second smooth  class  is determined solely by  its  lattice of open subgroups, and hence is invariant under topological isomorphism. To define this class, we need two  auxiliary notions,  the depth of open subgroups, and essential subgroups, Definition\ \ref{def_essential_intro}(3).

%\begin{definition}\label{def_depth} The set of open subgroups of $G$ under inclusions  form a lattice $(\mathcal{O}, \leq)$. Let $\lessdot$ be the resulting covering relation of $(\mathcal{O}_G, \leq)$, so that $H \lessdot H'$ iff $H \leq H'$, $H \neq H'$ and there is no $H''$ such that $H < H'' < H'$.
%By \ref{evans_fact} below, if $H < G$,   there is a chain    and $H = H_0, ..., H_{n} = G$ such that  $H_i \lessdot H_{i+1}$, for every $i < n$. Let the depth of $H$ be 
%the least length of such a chain (by convention $G$ has depth $0$).
%\end{definition}
It is well known  that for an oligomorphic group $G$,   there  are only finitely many  open subgroups   above any given one; this follows from~\ref{evans_fact}(ii).
\begin{definition}\label{def_depth}  The {\em depth} of an open subgroup  $U$ is the minimum of  all lengths of  maximal chains leading from $U$ to  $G$.  By this definition, $G$ has depth $0$. 
\end{definition}

\begin{definition} \label{def_essential_intro} \begin{enumerate}[(a)] \item One says that an open subgroup  of $G$ is \emph{irreducible} if it has no proper open subgroup of finite index. 
\item We say that 
  an irreducible    subgroup $H$ of $G$ is \emph{almost essential} if  every open subgroup of $G$ contains a finite intersection of conjugates of~$H$. \item Such a subgroup   $H$ is \emph{essential} if it is almost essential, and   its  depth is minimal among the    almost essential    subgroups. \end{enumerate} \end{definition}

We provide some model theoretic context to the condition in the definition of being  almost essential in~\ref{def_essential_intro}. Note that it is known that an open subgroup of $G$  is irreducible iff  it is the pointwise stabiliser of some set $\acl(e)$ where $e \in M^{\mrm{eq}}$. 
%LEAVE THIS \gianluca{What I say below clarifies the picture?} 
%
%\gianluca{Let $H \leq G$ be open and write $H = G_e$. Now suppose there is $H' \leq H$ which is open and write $H' = G_{e'}$. As $H' \leq H$, then $e \in \mrm{dcl}(e')$.
	%%
	%Furthermore, it seems to me that $H'$ has finite index in $H$ iff $e' \in \mrm{acl}(e)$. } \andre{yes, agreed}
%
%\gianluca{So indeed it would seem to me that is $H_* = G_{{e_*}}$ is smallest of finite index in $H$, then we have that for every $e' \in \mrm{acl}(e)$ we have that $e' \in \mrm{dcl}(e_*)$ and $e_* \in \mrm{acl}(e)$. Doesn't this mean that $H_* = G_{(\mrm{acl}(e))}$? All closure are taken in $M^{\mrm{eq}}$.} \andre{yes, sounds right... so irreducible means pointwise stabiliser of some $\acl(e)$, closure taken in $M^\mrm{eq}$. To actually put this we'd need to at least have a formal proof somewhere, to be produced if asked for (like logic blog). Did you take the notion of  irreducible from someone else's paper? } \gianluca{This is 100\% known to Evans but probably a reference does not exists. The term irreducible is used in \cite[Section~3.4]{Evans.etal:97} but not for general subgroups only for the group $G$ itself. I think that in some papers of Evans it might be used for all subgroups. But please don't ask me to dig it. I think that this is basic and we can just say: notice that irreducible means pointwise stabiliser of some $\acl(e)$. Anyway, you decide. All is good for me.}

\begin{proposition} \label{modelth background} Let  $U$ be an open subgroup of  $G = \mrm{Aut}(M)$. Write $ U = G_e$ for $e \in M^{\mrm{eq}}$, and let $S \subseteq M^{\mrm{eq}}$ be the $G$-orbit containing $e$. %So the conjugates of H are the stabilisers of the points of S. 
	The following are equivalent.
	\begin{enumerate}[(i)]
		\item Every open subgroup of $G$ contains a finite intersection of conjugates of $U$.
		\item The permutation group induced by $G$ on $S$ is faithful and closed.
		\item The structure induced by $M^{\mrm{eq}}$ on $S$ is bi-interpretable with $M$.  \end{enumerate} \end{proposition}

\begin{proof}   (i) $\to$ (ii). Firstly, the kernel of the   map $r$  in (ii) is contained in every open subgroup by (i), and so must be trivial. Secondly, for  every $a \in M$ there is a finite $B \subseteq S$ such that $a \in \dcl_{M^{\mrm{eq}}}(B)$. This implies that the group $r(G)$ induced on $S$ is closed: if $r(g_n)$ ($n \in \omega$) is a convergent sequence in the range  of $r$, then the $g_n$   converge and the limit $g$ is in $G$; so the $r(g_n)$ converge to $r(g)$.
	
	\smallskip \noindent   (ii) $\to$  (i).   $r$ gives a continuous isomorphism between the Polish groups $G$ and $r(G)$, so it is a homeomorphism. Thus, if $H \leq G$ is open, then $r(H)$ is open and so contains $G_{(B)}$ for some finite subset $B$ of $S$. 
	
	\smallskip \noindent   (iii) $\to$ (ii)  follows from  the above mentioned correspondence of topological isomorphisms and bi-interpretations \cite{biinte}. 
	
	\n   (ii) $\to$  (iii). The structure $N$ induced by $M^{\mrm{eq}}$ on $S$ is bi-definable with the structure obtained by  naming all the  orbits of  the action of $G$ on $S$. So, if the   map $r : G \to \Sym(S)$ is injective and has closed range, then $\mrm{Aut}(N) \cong G$.
\end{proof}

We are now ready to prove the statements of the introduction on groups with finitely many conjugacy classes of essential subgroups. 
\begin{lemma} \label{lem:essential}  The hypotheses  of Theorem~\ref{second_th-rec} are satisfied in the  following  setting:   

\begin{itemize}
	\item
	  $\+ V$ is the class of oligomorphic groups with finitely many essential subgroups up to conjugacy;   \item  $P(U,G)$ holds if $G \in \+ V$ and $U  $ is an essential subgroup.  \end{itemize}
\end{lemma}

 \begin{proof}  Let $P_0(U,G)$ express that $U$ is an essential subgroup of an oligomorphic group $G$. 
 	To  show that the binary relation $P_0$    is Borel, we use the notation from   the proof of~\ref{second_th-rec}(1), in particular   the Borel function $G\mapsto \+ W(G)$ and the associated assignment $G \mapsto \nu_G$ where $\nu_G$ is a bijection $\omega \to \+ W(G)$. There is an  $L_{\omega_1, \omega}$-formula $\phi$   with one free variable expressing  in  $\+ W(G)$ that a subgroup is essential; indeed,    the three conditions in Definition~\ref{def_essential_intro}(a)-(c) are expressed  by quantifying over finite collections of objects. Given  an open subgroup $U$ of  the oligomorphic group $G$,  we have that   $U$ is essential in $G$   iff $\+ W(G)\models \phi(U)$; this condition is Borel using the identification of $\omega$ and $\+ W(G)$ uniformly provided by the maps $\nu_G$.

\smallskip \noindent 	Next   we show that the class $\+ V$ is Borel.   As above, since $P_0$ is a Borel property, the map $G \mapsto \+ W^{P_0}(G)$ is Borel,  where $\+ W^{P_0}(G)$ is $\+ W(G)$ extended by a unary predicate  saying that a subgroup satisfies $P_0$.    Conjugacy of subgroups $U,V$ can be expressed in $\+ W(G)$ via the formula $\exists A [ U = A \cdot A^{-1}  \land V= A^{-1} \cdot  A]$. So  there is another $L_{\omega_1, \omega}$-sentence saying that there are finitely many conjugacy classes of essential subgroups.

 \smallskip \noindent 	
 $P$ is clearly invariant and sub-basic in the sense of Definition~\ref{def_sub-basic}.    \end{proof} 

 Recall from the introduction that an oligomorphic  group   is \emph{$\mathfrak{G}$-finite}   if  each open subgroup  contains a least open subgroup of finite index. Borelness of this   class  can be shown in a similar way as above; we omit the proof. 
   \begin{proposition} \label{there_are_essential}   Let $G$ be $\mathfrak G$-finite. Then $G$ has an essential subgroup.  \end{proposition} 

\begin{proof}   
By~\ref{1-transitive_fact} we can assume  that $G = \mrm{Aut}(M)$ and $G$ has a single $1$-orbit. For each $a \in M$, let $H(a) \leq G_{a}$ be the smallest open subgroup of $G$   that is a subgroup of  $G_{a}$ with  finite index. Note that  $H(a)$ is irreducible in the sense of  Definition~\ref{def_essential_intro}(a).
	Since each open subgroup of $G$ contains the pointwise stabiliser of a finite set,  each group of the form $H(a)$ is almost essential in  the sense of~\ref{def_essential_intro}. Hence  there is an almost essential subgroup of minimal depth, which is then essential by definition. 
\end{proof}

\subsection{Smoothness for weak elimination of imaginaries} \label{s:wei section}

The following is equivalent to the standard definition of weak elimination of imaginaries;  see e.g.\ \cite[Lemma~1.3]{evans_hrushovski} for a definition as well as a proof of the equivalence.

	\begin{definition}\label{very_weak} Let $M$ be $\aleph_0$-categorical. One says that $M$ has  \emph{weak elimination of imaginaries (w.e.i.)} if for every open $U \leq \mrm{Aut}(M)$ there is a unique finite algebraically closed set $A \subseteq M$ such that $G_{(A)} \leq U \leq G_{\{A\}}$. 
%	We say that $M$ has very weak elimination of imaginaries (vwei) if there are $n < \omega$ and sorts $S_1, ..., S_n$ in $M^{\mrm{eq}}$ such that for $N:=M \cup S_1 \cup \cdots \cup S_n $, the reduct of $M^{\mrm{eq}}$ to $N $ has weak elimination of imaginaries (we also allow the case $n = 0$ in which case $M = N$).  
\end{definition}

In the following we reprove the main result of~\cite[Th 1.1]{reconstruction2} in our framework (relying on some technical lemmas in~\cite{reconstruction2}),  and use this to obtain  new information about the   groups  $\Aut(G)$ where  $G= \Aut(M)$ and $M$ has w.e.i. 
\begin{corollary}[to Theorem~\ref{second_th-rec}] \label{weak_elimination_cor}  The class $\+ C$ of automorphism groups of $\aleph_0$-categorical structures that have w.e.i.\  is smooth~\cite{reconstruction2}. Furthermore, if $M$ has   w.e.i.\ and $G= \Aut(M)$, then  
 $\mrm{Aut}(G)$ is quasi-oligomorphic and $\Out(G)$ is profinite.
\end{corollary}

\begin{proof} $\+ C$ is Borel by~\cite[2.18]{reconstruction2}. It suffices to show that $\+ C$ is a subclass of  the class $\+ V$ from Lemma~\ref{lem:essential}.  Let $G= \Aut(M)$ where $M$ has w.e.i.

	\smallskip \noindent  To verify  that $G$ has an essential subgroup, by~\ref{there_are_essential} it suffices  to show that 
  $G$ is $\mathfrak G$-finite. To see this, suppose a finite set $A \subseteq M$ is   algebraically closed. We claim that if $V \subseteq G_{(A)}$ is open and of finite index in $G_{(A)}$  then $V = G_{(A)}$. We have $V = G_e$ for some imaginary $e$. Let $B = \mrm{acl}(e) \cap M$. Since $M$ has  w.e.i.\ we have $G_{(B)} \leq G_{(e)} \leq G_{\{B\}}$. So   $A \subseteq B$ (as $G_{(B)} \leq G_{(A)}$) and $e \in \mrm{acl}(A)$ (as $G_{(e)}$ is of finite index in $G_{(A)}$). It follows that $A = B$. 
 
\smallskip \noindent The irreducible subgroups of $G$ correspond precisely to the pointwise stabilisers of finite algebraically closed subsets of $M$ by~\cite[2.11]{reconstruction2}. 
 For any fixed depth $k < \omega$ there are only finitely many conjugacy classes of such stabilisers  by~\cite[2.17]{reconstruction2}.  A fortiori, there are only finitely many conjugacy classes of essential subgroups.   %Finally, we want to show that:
%\begin{enumerate}[(A)]
%	\item the group $\mrm{Aut}_{\mrm{top}}(G)$ is oligomorphic;
%	\item $\mrm{Aut}_{\mrm{top}}(G)/\Inn(G)$ is finite.
%\end{enumerate}
%Item (A) is by \ref{the_first_crucial_lemma}(\ref{crucial_item2}) as under the assumption of very weak elimination of imaginaries the structure $\mrm{Aut}(\+ C_G)$ is $\aleph_0$-categorical. Concerning item (B), that is by \ref{the_first_crucial_lemma}(\ref{crucial_item3}) as $\mrm{Aut}(M)$ is (modulo identifications) simply the subgroup of $\mrm{Aut}(\+ C_G)$ which fixes each equivalence class of the equivalence relations $L_n$'s from \ref{def_expanded_12_extra}(\ref{def_expanded_12_binary}).
\end{proof}

\end{document}